\documentclass[12pt,twoside]{article}
\usepackage{amsmath,amsfonts,amssymb}

\usepackage{color}
\usepackage{ifpdf}
\usepackage{graphicx}
\usepackage{upgreek}
\usepackage{amsbsy}
\usepackage{amssymb}
\usepackage{amsfonts}
\usepackage{amsmath}
\usepackage{amsopn}
\usepackage{amsthm}
\usepackage[english]{babel}
\usepackage[applemac]{inputenc}

\newtheorem{theorem}{Theorem}

\newtheorem{definition}[theorem]{Definition}

\newtheorem{lemma}[theorem]{Lemma}

\newtheorem{proposition}[theorem]{Proposition}
\theoremstyle{definition}
\newtheorem{remark}[theorem]{Remark}
\newtheorem{example}[theorem]{Example}
\newtheorem{remarks}[theorem]{Remarks}
\newcommand{\sign}{\textrm{sgn}}
\title{On L\'evy processes conditioned to avoid zero}
\author{Henry Pant\'i \thanks{Facultad de Matem\'aticas, Universidad Aut\'onoma de Yucat\'an. Anillo Periférico Norte, Tablaje Cat. 13615, Colonia Chuburná Hidalgo Inn, Mérida Yucatán. E-mail: henry.panti@correo.uady.mx.} \thanks{Centro de Investigaci\'on en Matem\'aticas (CIMAT A.C.). Calle Jalisco s/n, 36240 Guanajuato, Guanajuato M\'exico. E-mail: henry@cimat.mx.}}
\date{\today}
\begin{document}
\maketitle

\begin{abstract}
The purpose of this paper is to construct the law of a L\'evy process conditioned to avoid zero, under mild technicals conditions, two of them being that the point zero is regular for itself and the Lévy process is not a compound Poisson process. Two constructions are proposed, the first lies on the method of $h$-transformation, which requires a deep study of the associated excessive function; while in the second it is obtained by conditioning the underlying L\'evy process to avoid zero up to an independent exponential time whose parameter tends to $0.$ The former approach generalizes some of the results obtained by Yano \cite{Yano10} in the symmetric case and recovers some of main results in Yano's work \cite{Yano13}, while the latter is reminiscent of \cite{Chaumont-Doney05}. We give some properties of the resulting process and we describe in some detail two examples: alpha stable and spectrally negative Lévy processes.  
\end{abstract}



\section{Introduction}

The aim of this work is to construct L\'evy processes conditioned to avoid zero. This question is relevant only when $0$ is non-polar. Then the event ``not hitting zero'' has zero probability and hence a standard analytical approach consists on finding an adequate excessive function for the process killed at the first hitting time of zero and then use Doob's $h$-transformation technique. A good understanding of the associated excessive function allows us to establish analytical and pathwise  properties of the constructed process. This is the approach that has been used by Yano \cite{Yano10}, under the assumption that the L\'evy process is symmetric. The harmonic function obtained by Yano in \cite{Yano10} has been used to construct L\'evy processes conditioned to avoid zero in the symmetric case. Recently, Yano carried a similar study for asymmetric L\'evy processes \cite{Yano13}. The author in \cite{Yano13} gives some additional conditions in order to prove the existence of the harmonic function and he obtains an expression for it. Using a different approach, our work generalizes the results obtained in \cite{Yano10} and recovers those ones concerning to harmonic functions obtained in \cite{Yano13}. Moreover, we propose a probabilistic approach for constructing L\'evy processes conditioned to avoid zero, it relies on the idea that the construction can be performed by conditioning the process not to hit zero up to an independent exponential time with parameter $q$, and then make $q\to 0,$ so that the conditioning takes effect on the process all over the time interval $[0,\infty).$ This is a generic approach that has been used in several contexts. See for instance Chaumont and Doney~\cite{Chaumont-Doney05} and the reference therein, where the case of L\'evy processes conditioned to stay positive is investigated. We will prove that in our setting this procedure gives a non-degenerate limit and that this and the construction via Doob's $h$-transformation technique coincide.

The paper is organized as follows. In Section 2 the main results are stated. Some notations are introduced in section 2.1 and main results are stated in section 2.2. Section 3 is divided in two parts. The first part concerns the study of a sequence that defines as a limit the invariant function, some preliminary results and their proofs are given in this section. The second part is devoted to the study of an auxiliary needed function in most of the proofs in Section 4. Section 4 is devoted to prove the main theorems. In Section 5 two examples are studied where it is possible to compute explicitly the invariant function: the alpha-stable process and that of spectrally negative L\'evy process, i.e., processes with no positive jumps.

\section{Preliminaries and main results} \label{preliminaries}
\subsection{Notation} \label{notation}

Let $\mathcal{D}[0,\infty)$ be the space of c\`{a}dl\`{a}g paths $\omega: [0,\infty) \rightarrow \mathbb{R}\cup \{\Delta\}$ with lifetime $\zeta(\omega) = \inf\{s: \omega_s=\Delta\}$, where $\Delta$ is a cemetery point. The space $\mathcal{D}[0,\infty)$ is endowed with Skorohod's topology and its Borel $\sigma$-field, $\mathcal{F}.$  Moreover, let $\mathbb{P}$ be a reference probability measure on $\mathcal{D}[0,\infty),$ under which the coordinate process $X=(X_{t}, t\geq 0)$ is a L\'evy process. We will denote by $(\mathcal{F}_{t}, t\geq 0)$ the completed, right continuous filtration generated by $X$. As usual $\mathbb{P}_{x}$ denotes the law of $X+x,$ under $\mathbb{P},$ for $x\in\mathbb{R}.$ We have $\mathbb{P} = \mathbb{P}_0$ by definition. We will denote by $\theta$ the shift operator and by $k$ the killing operator, i.e., for $\omega \in \mathcal{D}[0,\infty)$, $\theta_t\omega(s) = \omega(s+t)$, $s\geq 0$, and
  $$
  k_t\omega(s) = \left\{
  \begin{tabular}{ll}
  $\omega(s)$, & $s<t$, \\
   & \\
  $\Delta$, & $s\geq t$.
  \end{tabular}
  \right.
  $$
For $t\geq 0$, we use $X\circ \theta_t$, $X\circ k_t$ to denote the functions in $\mathcal{D}[0,\infty)$ given by $\theta_t\omega(\cdot)$ and $k_t\omega(\cdot)$, $\omega\in\mathcal{D}[0,\infty)$, respectively. Throughout the paper $\psi:\mathbb{R}\to\mathbb{C}$ will denote the characteristic exponent of $(X,\mathbb{P})$, which is defined by
\begin{equation} \label{eqaux9-2}
\psi(\lambda) = -\frac{1}{t}\log(\mathbb{E}[e^{i\lambda X_t}]) = ia\lambda + \frac{\sigma^2}{2}\lambda^2 + \int_{\mathbb{R}}(1-e^{i\lambda x} + i\lambda x\mathbf{1}_{\{|x|<1\}})\pi(dx), \quad \lambda\in \mathbb{R},
\end{equation}
where $a\in \mathbb{R}$, $\sigma\geq 0$ and $\pi$ denotes the L\'evy measure, i.e., $\pi$ is a measure satisfying $\pi(\{0\})=0$ and $\int_{\mathbb{R}} (1\wedge x^2)\pi(dx)<\infty$. We denote by $P_t$ and $U_q$ the transition kernel at time $t$ and the $q$-resolvent of the process $(X,\mathbb{P})$. 

{\bf We assume throughout the paper that}
\begin{enumerate}
\item[\textbf{H.1}] The origin is regular for itself. 

\item[\textbf{H.2}] $(X,\mathbb{P})$ is not a compound Poisson process.

\item[\textbf{H.3}] The characteristic exponent $\psi$ satisfies
\begin{equation} \label{addcond}
\int_\mathbb{R} \left( \frac{1}{q + Re(\psi(\lambda))} \right) d\lambda <\infty, \quad q>0.
\end{equation}

\item[\textbf{H.4}] The following integrability condition holds
\begin{equation} \label{addcond2}
\int_\mathbb{R} \left|Re\left( \frac{1-e^{i\lambda}}{\psi(\lambda)} \right)\right| d\lambda <\infty.
\end{equation}
\end{enumerate}

We quote the following classical result that provides an equivalent way to verify the conditions \textbf{H.1} and \textbf{H.2} in terms of the characteristic exponent $\psi$.

\begin{theorem}[\cite{Bretagnolle71} and \cite{Kesten}] \label{theoKesten}
The conditions {\bf H.1} and {\bf H.2} are satisfied if and only if
\begin{equation*} 
\int_\mathbb{R} Re\left( \frac{1}{q+\psi(\lambda)} \right) d\lambda <\infty, \quad q>0
\end{equation*}
and
\begin{equation*} 
\textit{either} \quad \sigma\neq 0 \quad \textit{or} \quad \int_{|x|<1} |x|\pi(dx) = \infty.
\end{equation*}
\end{theorem}

It is known that under these hypotheses, for any $q>0$, there exists a density for the resolvent kernel that we will denote by $u_q(x,y)$:  
\begin{equation*}
U_qf(x) = \int_{\mathbb{R}} u_q(x,y)f(y)dy, \quad x\in\mathbb{R},
\end{equation*}
for all bounded Borel functions $f$. The density $u_q(x,y)$ equals $u_q(y-x)$, where $u_q$ is a continuous function. We refer to chapter II in \cite{Bertoin} for a proof of these results. Furthermore, from the resolvent equation
\begin{equation*}
U_q - U_r + (q-r)U_qU_r = 0, \quad q,r>0,
\end{equation*}
it can be deduced that the family of functions $(u_q,q>0)$ satisfies, for all $q,r>0$ with $q\neq r$, 
\begin{equation} \label{equ666}
\int_\mathbb{R} u_q(y-x)u_r(z-y) dy = \frac{1}{q-r}[u_r(z-x)-u_q(z-x)], \quad \textnormal{for all } z,x\in \mathbb{R}.
\end{equation}

For $x\in \mathbb{R}$, let $T_x$ be the first hitting time of $x$ for $X$:
  $$T_x = \inf\{t>0: X_t=x\},$$
with $\inf\{\emptyset\} = \infty$. The process killed at $T_0$, $X^0=X\circ k_{T_0}$, is given by
  $$X_t^0 = \left\{
  \begin{tabular}{ll}
  $X_t$, & $t<T_0$, \\
   & \\
  $\Delta$, & $t\geq T_0$.
  \end{tabular}
  \right.
  $$
For every $x\in \mathbb{R}$, we will denote by $\mathbb{P}_x^0$ the law of the killed process $X^0$ under $\mathbb{P}_x$. We use the notation $P_t^0$ and $U_q^0$ for its transition kernel and $q$-resolvent, respectively. From \cite[Corollary 18, p. 64]{Bertoin}, it is known that,
\begin{equation} \label{equ777}
\mathbb{E}_x[e^{-qT_0}] = \frac{u_q(-x)}{u_q(0)}, \quad q>0, x\in \mathbb{R}.
\end{equation}
Hence, with the help of the following well known identity:
\begin{equation*}
U_qf(x) = U_q^0f(x) + \mathbb{E}_x[e^{-qT_0}]U_qf(0),
\end{equation*}
for all bounded Borel functions $f$ and $q>0$, we obtain the resolvent density for $X^0$, namely,
\begin{equation} \label{equ888}
u_q^0(x,y) = u_q(y-x) - \frac{u_q(-x)u_q(y)}{u_q(0)}, \quad x,y\in\mathbb{R}.
\end{equation}

By $\widehat{\mathbb{P}}_x$ we will denote the law of the dual process $\widehat{X}:=-X$ under $\mathbb{P}_{-x}$, $x\in\mathbb{R}$. We will use the notation $\widehat{\quad}$ to specify the mathematical quantities related to the dual process $\widehat{X}$. For instance, $(\widehat{P}_t,t\geq 0)$, $(\widehat{U}_q, q>0)$ are the semigroup and the resolvent of the process $\widehat{X}$, respectively. It is known that the name ``dual'' comes from the following duality identity. Let $f$, $g$ be nonnegative and measurable functions. Then, for every $t\geq 0$
\begin{equation*}
\int_{\mathbb{R}} P_tf(x)g(x) dx = \int_{\mathbb{R}} f(x)\widehat{P}_t g(x) dx
\end{equation*}
and for every $q>0$
\begin{equation*} 
\int_{\mathbb{R}} U_qf(x) g(x)dx = \int_{\mathbb{R}} f(x) \widehat{U}_q g(x) dx.
\end{equation*}
For the semigroup and $q$-resolvent of the killed process we have as a consequence of \emph{Hunt's switching identity} (see e.g. \cite[p. 47, Theorem 5]{Bertoin}):
\begin{equation*}
\int_{\mathbb{R}} g(x) P_t^0f(x) dx = \int_{\mathbb{R}} f(x)\widehat{P}_t^0 g(x) dx
\end{equation*}
and for every $q>0$
\begin{equation*}
\int_{\mathbb{R}} g(x)U_q^0f(x) dx = \int_{\mathbb{R}} f(x) \widehat{U}_q^0 g(x) dx.
\end{equation*}

We observe that $(\widehat{X},\widehat{\mathbb{P}})$ satisfies also the hypotheses \textbf{H.1} and \textbf{H.2}. Thus, for any $q>0$, there exists a continuous density $\widehat{u}_q$ of the resolvent $\widehat{U}_q$. Furthermore, $u_q$ and $\widehat{u}_q$ are related by the equation: $\widehat{u}_q(x) = u_q(-x)$, $x\in \mathbb{R}$. Thereby, for any $q>0$, $\widehat{\mathbb{E}}_x[e^{-qT_0}]$ and the density of $\widehat{U}_q^{0}$ can be written in terms of $u_q$ as follows
\begin{equation} \label{equ999}
\widehat{\mathbb{E}}_x[e^{-qT_0}] = \frac{u_q(x)}{u_q(0)}, \quad q>0, x\in\mathbb{R}
\end{equation}
and 
\begin{equation*} 
\widehat{u}_q^0(x,y) = u_q(x-y) - \frac{u_q(x)u_q(-y)}{u_q(0)}, \quad x,y\in\mathbb{R}.
\end{equation*}

Since the point zero is regular for itself, there exists a continuous local time at 0 (in fact, at any point $x\in\mathbb{R}$). We denote by $L = (L_t,t\geq 0)$ the local time at zero, which is normalized by $\mathbb{E}(\int_0^\infty e^{-t}dL_t)=1$, and by $n$ the excursion measure away from zero for $X$. The measure $n$ is carried by the set of excursions away from zero:
\begin{equation*}
\mathcal{D}^0 = \left\{\upepsilon \in \mathcal{D}[0,\infty): \upepsilon(t)\neq 0, 0<t<\zeta(\upepsilon), 0<\zeta(\upepsilon) \leq \infty \right\}.
\end{equation*}
A nice relation between the excursion measure $n$ and the Laplace transform of the law of $T_0$ under $\widehat{\mathbb{P}}_x$ can be found in \cite[Theorem 3.3]{Yano-Yano-Yor09} for L\'evy processes and in \cite[eq. (3.22)]{Fitzsimmons-Getoor06}, \cite[eq. (2.8)]{Chen-Fukushima-Ying07} for general Markov processes. This is stated as follows, let $f$ be a nonnegative measurable function, then
\begin{equation} \label{equ1111}
\int_0^\infty e^{-qt}n(f(X_t), t<\zeta)dt = \int_{\mathbb{R}} f(x) \widehat{\mathbb{E}}_x[e^{-qT_0}]dx.
\end{equation}
In particular, if $f\equiv 1$, 
\begin{equation} \label{equ1212}
\int_0^\infty e^{-qt}n(\zeta>t)dt = \frac{1}{qu_q(0)}, \quad q>0.
\end{equation}

\subsection{Main results} \label{construction}

\textbf{Throughout the rest of this paper, and unless otherwise stated, we assume that \textbf{H.1}--\textbf{H.4} are satisfied.} Under these assumptions we have our first main result.

\begin{theorem} 
\label{propofh}  
For $q>0,$ let $h_{q}$ denote the function defined by 
\begin{equation}\label{eq:hqdef}
h_q(x) = u_q(0)-u_q(-x), \quad q>0, x\in\mathbb{R}.
\end{equation}
Then, the identity
\begin{equation} \label{eq:hqidentity}
h_q(x) = \frac{\mathbb{P}_x(T_0>\mathbf{e}_q)}{n(\zeta>\mathbf{e}_q)}, \quad x\in \mathbb{R}, 
\end{equation}
holds, where $\mathbf{e}_q$ is an exponential random variable with parameter $q>0$ independent of $(X, \mathbb{P})$.  The limit $\lim_{q\to 0}h_q(x)$ exists for all $x\in \mathbb{R}$ and the function $h:\mathbb{R}\to\mathbb{R}$ defined by 
\begin{equation}\label{eq:hdef} 
h(x)=\lim_{q\rightarrow 0} h_q(x), \quad x\in \mathbb{R}, 
\end{equation} is such that
\begin{enumerate}
\item[(i)] for every $x\in\mathbb{R}$, $0\leq h(x)<\infty$ and it holds
\begin{equation} \label{equ2121}
h(x) = \frac{1}{2\pi} \int_{-\infty}^\infty Re\left( \frac{1-e^{i\lambda x}}{\psi(\lambda)} \right) d\lambda, \quad x\in\mathbb{R}.
\end{equation}

\item[(ii)] $h$ is a subadditive, continuous function, which vanishes at the point $x=0$.

\item[(iii)] $h$ is invariant with respect to the semigroup of the L\'evy process killed at $T_0$, i.e.,
\begin{equation*}
P_t^0h(x) = h(x), \quad t>0, x\in \mathbb{R};
\end{equation*}
and furthermore
\begin{equation*}
n(h(X_t), t<\zeta) = 1, \quad \forall\, t>0.
\end{equation*}
\end{enumerate}
\end{theorem}

\begin{remarks}
\begin{enumerate}
\item[(i)] Under the assumptions \textbf{H.1}, \textbf{H.2} and $(X,\mathbb{P})$ is symmetric, Yano \cite{Yano10} showed that the function $h$ defined by  
\begin{equation} \label{equ111}
h(x) = \lim_{q\rightarrow 0} [u_q(0) - u_q(x)], \quad x\in \mathbb{R}
\end{equation}
is a well defined invariant function for the semigroup of the L\'evy process killed at its first hitting time of zero. In the same paper it is shown that the function $h$ can be expressed in terms of the characteristic exponent of $X$ as 
\begin{equation} \label{equ222}
h(x) = \frac{1}{2\pi} \int_{\mathbb{R}} \frac{1-\cos\lambda x}{\theta(\lambda)} d\lambda, \quad x\in\mathbb{R},
\end{equation}
where $\theta(\lambda)=Re(\psi(\lambda))$.  Theorem~\ref{propofh} extends (\ref{equ111}) and (\ref{equ222}).

\item[(ii)] In \cite{Yano13} the formula (\ref{equ2121}) is obtained under the conditions \textbf{H.3} and that the functions $\theta = Re(\psi)$, $\omega = Im(\psi)$ have measurable derivatives on $(0,\infty)$ satisfying
   $$\int_0^\infty \frac{ (| \theta'(\lambda) | + | \omega'(\lambda) |)(\lambda^2 \wedge 1) }{ \theta(\lambda)^2 + \omega(\lambda)^2 } d\lambda <\infty.$$ 
Instead of the latter condition, we consider the hypothesis \textbf{H.4} which seems less restrictive.
\end{enumerate}
\end{remarks}

The proof of (i) and (ii) in Theorem~\ref{propofh} will be given in section \ref{spofh}, as a consequence of analogous results for the sequence of functions $(h_q)_{q>0}$. In order to establish (iii) and other results, and due to technical issues, we will introduce an auxiliary function $h^\ast$. The function $h^\ast$ dominates $h$ and satisfies some integrability conditions. This function, as its name indicates, will help us to prove the main results, acting as a dominating function in an application of the dominated convergence theorem. The function $h^\ast$ is closely related to the local time of the L\'evy process $(X,\mathbb{P}),$  namely, we have the expression
\begin{equation*}
h^\ast(x) = \mathbb{E}( L_{T_x} ) = \lim_{q\rightarrow 0} \mathbb{E}\left( \int_0^{T_x} e^{-qt} dL_t \right), \quad x\in \mathbb{R}.
\end{equation*}
The function $h^\ast$ arises as a particular case of a general function $h(\cdot, \cdot)$ defined by
  $$h(x,y) = \mathbb{E}_x( L_{T_y}^x ) = \mathbb{E}_0(L_{T_{y-x}}^0) = h(0,y-x) = h^\ast(y-x),$$
where $L_t^x$ denotes the local time at the point $x$ for the process $(X,\mathbb{P}_x)$. The function $h(\cdot, \cdot)$ has been used to establish continuity criteria for local times of L\'evy processes, see \cite{Barlow85, Barlow88} for this case and \cite{Eisenbaum-Kaspi07} for general Borel right Markov processes.

Besides, in the present context, both Yano's and our results can be seen as an extension of the theory of invariant functions for killed L\'evy processes that can be found in Section 23 of the treatise by Port and Stone \cite{Port-Stone71} on the potential theory for L\'evy processes in locally compact, non-compact, second countable Abelian groups. Detailing the relation with that paper would require us to introduce further notations and facts that will not be used later, so we do not provide further details.

Having constructed the invariant function $h,$ in the following definition we introduce the associated $h$-process. We will show that the resulting probability measures are such that the canonical process X never hits the point zero, and thus that we refer to them as the law of the L\'evy process conditioned to avoid zero. Theorem \ref{mainresult} below summarises these properties. Let $\mathcal{H}$ and $\mathcal{H}_0$ be the sets given by 
  $$\mathcal{H} = \{x\in \mathbb{R}: h(x)>0\}, \quad \mathcal{H}_0 = \mathcal{H}\cup\{0\}.$$ 
The law of the L\'evy process conditioned to avoid zero will be constructed on the set $\mathcal{H}_0$.

\begin{definition}
We denote by $(\mathbb{P}_x^\updownarrow, x\in \mathcal{H}_0)$ the $h$-transform of $(\mathbb{P}_x^0, x\in \mathcal{H}_0)$ associated to the invariant function $h$ defined in Theorem~\ref{propofh} (iii). That is, $(\mathbb{P}_x^\updownarrow, x\in \mathcal{H}_0)$ is the unique family of measures such that for $x\in\mathcal{H}_0$, 
  $$\mathbb{P}_x^\updownarrow(\Lambda) = \left\{ 
  \begin{tabular}{cc}
  $\displaystyle{\frac{1}{h(x)} \mathbb{E}_x^0(\mathbf{1}_{\Lambda}h(X_t))}$, & $\quad x\in \mathcal{H}$, \\
   & \\
  $n(\mathbf{1}_\Lambda h(X_t) \mathbf{1}_{\{ t<\zeta\}})$, & $x=0$,
  \end{tabular} 
  \right.
  $$
for all $\Lambda \in \mathcal{F}_t$, for all $t\geq 0$. We will refer to it as the law of $X$ conditioned to avoid $0$.
\end{definition}

\begin{remark}
Note that from this definition and Theorem \ref{propofh}, $\mathbb{P}_x^\updownarrow(T_0>t) = 1$, for all $t>0$, $x\in\mathcal{H}_0$. Hence, $\mathbb{P}_x^\updownarrow(T_0=\infty)=1$, for all $x\in\mathcal{H}_0$. 
\end{remark}

\begin{theorem} \label{mainresult}
The family of measures $(\mathbb{P}_x^\updownarrow)_{x\in \mathbb{R}}$ is Markovian and satisfies 
\begin{enumerate}
\item[(i)] $\mathbb{P}_x^\updownarrow(X_0=x)=1$, $\forall\, x\in\mathcal{H}_0$.

\item[(ii)] $\mathbb{P}_x^\updownarrow(T_0=\infty)=1$, $\forall\, x\in\mathcal{H}_0$.
\end{enumerate}
The semigroup associated to $(\mathbb{P}_x^\updownarrow)_{x\in\mathbb{R}}$ is given by
\begin{equation*}
P_t^\updownarrow(x, dy) := \frac{h(y)}{h(x)}P_t^0(x,dy), \quad x,y\in \mathcal{H}, \quad t\geq 0.
\end{equation*}
The entrance law under $\mathbb{P}_0^\updownarrow$ is given by
\begin{equation*}
\mathbb{P}_0^\updownarrow(X_t\in dy) = n(X_t\in dy, t<\zeta)h(y).
\end{equation*}
\end{theorem}

We propose an alternative construction of the law of the L\'evy process conditioned to avoid zero. Our construction is inspired from  \cite{Bertoin93, Chaumont96, Chaumont-Doney05, Chaumont-Doney08}, where L\'evy processes conditioned to stay positive are constructed. This construction is given in Theorems \ref{condavoidzero} and \ref{theo5} below.

The following theorem states that for $x\in \mathcal{H}$, $\mathbb{P}_x^\updownarrow$ is the limit, as $q\rightarrow 0$, of the law of the process $X$ under $\mathbb{P}_x$ conditioned to avoid zero, up to an independent exponential time with parameter $q>0$. Since an exponential random variable with parameter $q$ converges in distribution to infinity as its parameter converges to zero, then this result confirms that, starting at $x\in \mathcal{H}$, we can think of $X$ under $\mathbb{P}_x^\updownarrow$, as the process conditioned to avoid zero on the whole positive real line.

\begin{theorem} \label{condavoidzero}  
Let $\mathbf{e}_q$ be an exponential time with parameter $q>0$, independent of $(X, \mathbb{P})$. Then for any $x\in \mathcal{H}$, and $t>0$,
\begin{equation*}
\lim_{q\rightarrow 0} \mathbb{P}_x(\Lambda, t<\mathbf{e}_q \mid T_0>\mathbf{e}_q) = \mathbb{P}_x^\updownarrow(\Lambda), \quad \forall\, \Lambda\in \mathcal{F}_t.
\end{equation*}
\end{theorem}

In the case $x=0$, the law $\mathbb{P}_0^\updownarrow$ can also be obtained as a limit involving an independent exponential time. Before stating the result, we point out that for $s>0$, we will denote by $g_s = \sup\{ t\leq s: X_t=0\}$, the last zero of $X$ before time s.

\begin{theorem} \label{theo5}
Let $\mathbf{e}_q$ be an exponential time with parameter $q>0$, independent of $(X, \mathbb{P})$. Let $\mathbb{P}^{\mathbf{e}_q}$ be the law of $X\circ k_{\mathbf{e}_q-g_{\mathbf{e}_q}}\circ \theta_{g_{\mathbf{e}_q}}$ under $\mathbb{P}$. Then, for $t>0$,
\begin{equation*}
\lim_{q\rightarrow 0}\mathbb{P}^{\mathbf{e}_q} (\Lambda, t<\zeta) = \mathbb{P}_0^\updownarrow(\Lambda) = n(\mathbf{1}_\Lambda h(X_t) \mathbf{1}_{\{t<\zeta\}}), \quad \forall\, \Lambda\in \mathcal{F}_t.
\end{equation*}
\end{theorem}

Another important property of the $h$-process is its transiency. This is given in the following proposition.

\begin{proposition}[Transiency property] \label{transitorityofhprocess}
The process $(X, \mathbb{P}_x^\updownarrow)_{x\in\mathcal{H}_0}$ is transient.
\end{proposition}

In Lemma \ref{indoflpaz} we will prove that for any $x\in \mathcal{H}$, the point $x$ is regular for itself under $\mathbb{P}_x^\updownarrow$. Therefore, there exists a local time at any point $x\in\mathcal{H}$, and we will denote by $n_x^\updownarrow$ the excursion measure away from $x$ for the process $(X,\mathbb{P}_x^\updownarrow)$. In the following proposition we establish a relationship between the excursion measure away from zero for $(X,\mathbb{P})$ and the excursion measure away from $x$ for $(X,\mathbb{P}_x^\updownarrow)$, $x\in \mathcal{H}$.

\begin{proposition} \label{relexcmeasure}
For $x\in \mathcal{H}$, let $n_x^\updownarrow$ be the excursion measure out from $x$ for $(X, \mathbb{P}_x^\updownarrow)$ and $n$ the excursion measure out from zero for $(X, \mathbb{P})$. Then, for any measurable and bounded functional $H:\mathcal{D}^0 \rightarrow \mathbb{R}$, 
\begin{equation*}
n_x^\updownarrow \left( \int_0^\zeta H(\upepsilon_u, u<t) qe^{-qt} dt \right) = \frac{1}{h(x)} n \left( \int_0^\zeta H(\upepsilon_u + x, u<t)h(X_t+x)\mathbf{1}_{\{T_{-x}>t\}} qe^{-qt} dt \right). 
\end{equation*}
\end{proposition}

\section{Preliminary results} \label{proofssection}
\subsection{Some properties of $h_q$ and $h$} \label{spofh}
In order to prove the finiteness of $h$, we need the following lemma.

\begin{lemma} \label{lemmaaux}
Let $(X,\mathbb{P})$ be a L\'evy process with characteristic exponent $\psi$. Assume that $(X,\mathbb{P})$ 
satisfies the hypotheses {\bf H.1} and {\bf H.2}, then $\psi(\lambda)\neq 0$, for all $\lambda\neq 0$ and 
\begin{equation*}
\lim_{|\lambda|\rightarrow \infty} |\psi(\lambda)| = \infty. 
\end{equation*}
Furthermore, 
\begin{equation} \label{equ1313}
\int_{\mathbb{R}} (1\wedge \lambda^2) Re\left( \frac{1}{\psi(\lambda)} \right) d\lambda < \infty.
\end{equation}
\end{lemma}
\begin{proof}
The first part follows from the fact that $(X,\mathbb{P})$, because by assumption it is not a compound Poisson process (see e.g. \cite[Theorem 6.4.7]{Chung01}). Now, since $1/(1+\psi)$ is the Fourier transform of the integrable function $u_1$, then from the Riemann-Lebesgue theorem it follows that $\lim_{|\lambda|\rightarrow \infty} |\psi(\lambda)| = \infty$.

Using that $\lim_{|\lambda|\rightarrow \infty} |\psi(\lambda)| = \infty$, we deduce 
  $$Re\left( \frac{1}{\psi(\lambda)} \right) \sim Re\left( \frac{1}{1+\psi(\lambda)} \right), \quad |\lambda|\rightarrow \infty.$$
The latter and Theorem \ref{theoKesten} imply that for all $\lambda_0>0$,
\begin{equation} \label{equ1414}
\int_{|\lambda|>\lambda_0} Re\left( \frac{1}{\psi(\lambda)} \right) d\lambda <\infty.
\end{equation}
On the other hand, observe the elementary inequality
  $$Re\left( \frac{\lambda^2}{\psi(\lambda)} \right) = \lambda^2 \frac{ Re(\overline{\psi(\lambda)}) }{ | \psi(\lambda) |^2} = \lambda^2 \frac{ Re(\psi(\lambda)) }{ | \psi(\lambda) |^2} \leq \frac{\lambda^2}{Re(\psi(\lambda))}.$$
Notice that in the last inequality we use that $Re(\psi(\lambda))>0$, which easily follows from (\ref{eqaux9-2}). Thus,
\begin{eqnarray*}
\left[ Re\left( \frac{ \lambda^2}{ \psi(\lambda)} \right) \right]^{-1} 
&\geq& \frac{ Re(\psi(\lambda)) }{ \lambda^2 }  \\
&\geq& \sigma^2 + \int_{|y|<1} \frac{(1-\cos\lambda y)}{\lambda^2} \pi(dy) \\
&\longrightarrow& \sigma^2 + \int_{|y|<1} y^2 \pi(dy)>0, \quad \textnormal{as } \lambda\rightarrow 0. 
\end{eqnarray*}
The latter limit implies that there exists a $\lambda_0$ such that, 
\begin{equation} \label{equ1515}
Re\left( \frac{ \lambda^2}{ \psi(\lambda)} \right) \leq C, \quad \textnormal{for all } |\lambda|<\lambda_0,
\end{equation}
for some positive constant $C$. Then, from (\ref{equ1414}) and (\ref{equ1515}), we obtain (\ref{equ1313}).
\end{proof}

\begin{remark} \label{intcond}
Assume that  {\bf H.1}, {\bf H.2} and {\bf H.3} hold. On the one hand, by Lemma \ref{lemmaaux}, $Re(\psi(\lambda)) \to \infty$ as $|\lambda| \to \infty$, and hence
  \[
  \frac{1}{Re(\psi(\lambda))} \sim \frac{1}{1+Re(\psi(\lambda))}, \quad |\lambda| \to \infty.
  \]
This shows that for all $\lambda_0>0$,
\begin{equation*} 
\int_{|\lambda|>\lambda_0} \frac{1}{Re(\psi(\lambda))} d\lambda <\infty.
\end{equation*}
On the other hand, if we proceed as in last part of the proof of Lemma \ref{lemmaaux}, it can be proved that there exists a $\lambda_0$ such that, 
\begin{equation*} 
\frac{ \lambda^2}{ Re(\psi(\lambda))} \leq C, \quad \textnormal{for all } |\lambda|<\lambda_0,
\end{equation*}
for some positive constant $C$. Thus, under the assumptions  {\bf H.1}, {\bf H.2} and {\bf H.3}, it holds
\begin{equation*} 
\int_{\mathbb{R}}  \frac{1\wedge \lambda^2}{Re(\psi(\lambda))}  d\lambda < \infty.
\end{equation*}
\end{remark}

To establish some properties of $h$, we recall the representation of $h_q$ given in (\ref{eq:hqidentity}). Let $\mathbf{e}_q$ be an exponential random variable with parameter $q>0$ and independent of $(X,\mathbb{P})$. Using (\ref{equ777}) and (\ref{equ1212}), we can write 
\begin{equation} \label{equ1616}
\begin{split}
h_q(x) &= u_q(0)(1-\mathbb{E}_x(e^{-qT_0})) \\ 
&= \frac{\mathbb{P}_x(T_0>\mathbf{e}_q)}{n(\zeta>\mathbf{e}_q)},
\end{split}
\end{equation}
where
\begin{equation*} 
n(\zeta>\mathbf{e}_q) = \int_0^\infty qe^{-qt}n(\zeta>t)dt = \frac{1}{u_q(0)}.
\end{equation*}

From the expression (\ref{equ1616}) follows that if $(X,\mathbb{P})$ is transient, then
\begin{equation} \label{htransient}
h(x) = \lim_{q\to 0} h_q(x) = \kappa^{-1} \mathbb{P}_x(T_0=\infty),
\end{equation}
where $\kappa = \lim_{q\to 0} \frac{1}{u_q(0)}$. Furthermore, (\ref{equ1616}) helps us to prove the following Lemma, which summarizes some important properties of the sequence $(h_q)_{q>0}$.

\begin{lemma} \label{propertiesofhq}
For every $q>0$, the function $h_q$ is subadditive on $\mathbb{R}$ and it is excessive for the semigroup $(P_t^0,t\geq 0)$.
\end{lemma}
\begin{proof}
By Proposition 43.4 in \cite{Sato}, we have that for any $q>0$ and $x,y\in \mathbb{R}$,
\begin{equation} \label{equ1717}
\mathbb{E}_{x+y}(e^{-qT_0}) \geq \mathbb{E}_x(e^{-qT_0})\mathbb{E}_y(e^{-qT_0}).
\end{equation}
Now, since
\begin{equation*}
(1-\mathbb{E}_x(e^{-qT_0}))(1-\mathbb{E}_y(e^{-qT_0}))\geq 0,
\end{equation*}
then using (\ref{equ1717}), it follows
\begin{equation*}
1-\mathbb{E}_x(e^{-qT_0}) + 1-\mathbb{E}_y(e^{-qT_0}) \geq 1 - \mathbb{E}_{x+y}(e^{-qT_0}).
\end{equation*}
Hence, by (\ref{equ1616})
\begin{equation*}
h_q(x+y)\leq h_q(x)+h_q(y), \quad x,y\in\mathbb{R}.
\end{equation*}
This shows that $h_q$ is subadditive on $\mathbb{R}$.

In order to show that $h_q$ is excessive for $P_t^0$, we observe that by the Markov property 
\begin{equation} \label{equ1818}
\mathbb{P}_x(T_0>t+\mathbf{e}_q) = \mathbb{E}_x \left( \mathbf{1}_{\{ T_0> t+\mathbf{e}_q \}} \right) = \mathbb{E}_x\left( \mathbb{P}_{X_t}(T_0>\mathbf{e}_q) \mathbf{1}_{\{t<T_0\}} \right).
\end{equation}
The identities (\ref{equ1818}) and (\ref{equ1616}) imply
\begin{eqnarray*}
\mathbb{E}_x(h_q(X_t),t<T_0) &=& \mathbb{E}_x \left( \frac{ \mathbf{1}_{\{ T_0> t+\mathbf{e}_q \}} }{ n(\zeta > \mathbf{e}_q) }  \right) \\
&\leq& \mathbb{E}_x \left( \frac{ \mathbf{1}_{\{ T_0> \mathbf{e}_q \}} }{ n(\zeta > \mathbf{e}_q) }  \right) \\
&=& h_q(x).
\end{eqnarray*}
The above expression also implies that $\lim_{t\to 0}\mathbb{E}_x(h_q(X_t),t<T_0)=h_{q}(x),$ for $x\in \mathbb{R}.$
This shows that $h_q$ is excessive for the semigroup $(P_t^0,t\geq 0)$.
\end{proof}

Before we proceed to the proof of (i) and (ii) in Theorem~\ref{propofh} we make a technical remark. 
\begin{remark}\label{techremark}
Proceeding as in the proof of Theorem 19 p. 65 in \cite{Bertoin}, it can be shown that
\begin{equation} \label{equ1919}
u_q(x) = \frac{1}{2\pi} \int_\mathbb{R} Re \left( \frac{e^{-i\lambda x}}{q+\psi(\lambda)} \right) d\lambda, \quad x\in\mathbb{R}.
\end{equation}
Then, 
\begin{equation*}
2u_q(0) - [u_q(x)+u_q(-x)] = \frac{1}{\pi} \int_\mathbb{R} (1-\cos \lambda x) Re\left( \frac{1}{q + \psi(\lambda)} \right) d\lambda.
\end{equation*}
On the other hand, making use of the inequality $|1-\cos b|\leq 2(1\wedge b^2)$ and (\ref{equ1313}), we obtain
\begin{equation*} 
\int_\mathbb{R} (1-\cos \lambda x) Re\left( \frac{1}{\psi(\lambda)} \right) d\lambda < \infty, \quad x\in \mathbb{R}.
\end{equation*}
Furthermore, since $Re(\psi(\lambda))>0$, 
  \[
  \left| (1-\cos \lambda x) Re\left( \frac{1}{ q + \psi(\lambda) } \right) \right| \leq \frac{2(1\wedge (\lambda x)^2) }{Re(\psi(\lambda))}.
  \]
The right-hand function of latter inequality is integrable by Remark \ref{intcond}. Thus, by the dominated convergence theorem, for all $x\in\mathbb{R}$,  
\begin{equation} \label{equ2020}
\lim_{q\rightarrow 0} (2u_q(0) - [u_q(x)+u_q(-x)]) = \frac{1}{\pi} \int_\mathbb{R} (1-\cos \lambda x) Re\left( \frac{1}{\psi(\lambda)} \right) d\lambda
\end{equation}
is finite.
\end{remark}

\begin{proof}[Proof of (i) and (ii) in Theorem~\ref{propofh}]
That $h$ is subadditive and excessive follows from Lemma \ref{propertiesofhq} (since these properties are preserved under limits of sequences of functions).

Using (\ref{equ1919}), we infer the following expression for $h_{q},$
\begin{equation*}
h_q(x) = \frac{1}{2\pi} \int_{-\infty}^\infty Re\left( \frac{1-e^{i\lambda x}}{q+\psi(\lambda)} \right) d\lambda.
\end{equation*}
We will deduce therefrom the identity (\ref{equ2121}). For that end, let $x\in \mathbb{R}$ fixed, and denote by $g_x:\mathbb{R}\to\mathbb{R}$ the integrand in (\ref{equ2121}), that is, 
  \[
  g_x(\lambda) = Re\left( \frac{1-e^{i\lambda x}}{ \psi(\lambda) } \right), \quad \lambda \in \mathbb{R}.
  \]
By hypothesis {\bf H.4}, we have that $g_{x}\in L_{1},$ for any $x\in \mathbb{R}$. Furthermore, $g_x(\lambda)$ is related to the integrand that defines $h_{q}$ through the equality
  \[
  Re\left( \frac{1-e^{i\lambda x}}{ q + \psi(\lambda) } \right) = g_x(\lambda) \frac{ |\psi(\lambda) |^2 }{ |q +\psi(\lambda)|^2} + q \frac{1-\cos(\lambda x) }{|q +\psi(\lambda)|^2}, \quad q>0, \lambda \in \mathbb{R}. 
  \]
Since for $q>0$, $\lambda\in \mathbb{R}$, we have
  \[
   \frac{ |\psi(\lambda) |^2 }{ |q +\psi(\lambda)|^2} \leq 1
  \]
and
\begin{align*}
|q +\psi(\lambda)|^2 \geq 2 qRe(\psi(\lambda)),
\end{align*}
we derive the following domination
\begin{equation*}
  \left| Re\left( \frac{1-e^{i\lambda x}}{ q + \psi(\lambda) } \right) \right| \leq |g_x(\lambda)| +  \frac{(1\wedge (\lambda x)^2) }{  Re(\psi(\lambda))}, \quad q>0, \lambda \in\mathbb{R}. 
\end{equation*}
The Remark \ref{intcond} allow us to ensure that $\int^{\infty}_{-\infty}\frac{1\wedge (\lambda x)^2 }{ Re(\psi(\lambda))}d\lambda <\infty.$ This observation together with the fact that $g_x\in L_{1}$ permit us to apply the dominated convergence theorem to establish that for any $x\in\mathbb{R}$ the limit of $h_{q}(x)$ as $q\to 0$ exists, is finite and furthermore $\lim_{q\to 0}h_{q}(x)=h(x)$ with $h$ as defined in (\ref{equ2121}).

Finally we prove the continuity of $h$. We note that for all $q>0$, $x\in \mathbb{R}$,
\begin{equation*}
h_q(x) \leq 2u_q(0) - [u_q(x) + u_q(-x)] = \frac{1}{\pi}\int_{-\infty}^\infty (1-\cos\lambda x ) Re\left( \frac{1}{q+\psi(\lambda)} \right) d\lambda.
\end{equation*}
Then, by (\ref{equ2020}), 
\begin{equation} \label{equ2222}
h(x) \leq \frac{1}{\pi} \int_{-\infty}^\infty (1-\cos\lambda x ) Re\left( \frac{1}{\psi(\lambda)} \right) d\lambda, \quad \forall\, x\in\mathbb{R}.
\end{equation}  
Now, observe that 
  $$(1-\cos \lambda x) Re\left( \frac{1}{\psi(\lambda)} \right) \leq 2(1\wedge \lambda^2) Re\left( \frac{1}{\psi(\lambda)} \right), \quad |x|\leq 1, \lambda\in \mathbb{R}.$$
Then, by (\ref{equ1313}) and the dominated convergence theorem, it follows that
\begin{equation*}
\lim_{x\rightarrow 0} \frac{1}{\pi}\int_{-\infty}^\infty (1-\cos\lambda x) Re\left( \frac{1}{\psi(\lambda)} \right) d\lambda = 0.
\end{equation*}
Hence, by (\ref{equ2222}), $\lim_{x\rightarrow 0} h(x)=0$. This proves that $h$ is continuous at zero. Furthermore, since $h$ is subadditive on $\mathbb{R}$, the continuity of $h$ at the point zero implies the continuity on the whole real line (see e.g. \cite[Theorem 6.8.2]{Hille-Phillips57}).
\end{proof}

To end this section, we establish the behaviour of $h$ at infinity.

\begin{lemma} \label{limitbehaviourofh}
Let $\kappa := \lim_{q\rightarrow 0} \frac{1}{u_q(0)}.$ We have the following 
\begin{enumerate}
\item[(i)]  
Suppose that $X$ is transient. If $0<\mu:=\mathbb{E}(X_{1})\leq \infty$, then
  $$\lim_{x\to \infty} h(x) = \frac{1}{\kappa}, \quad \lim_{x\to -\infty} h(x) = \frac{1}{\kappa} - \frac{1}{\mu};$$
while if $-\infty\leq \mu <0$, then
  $$\lim_{x\to \infty} h(x) = \frac{1}{\kappa} + \frac{1}{\mu}, \quad \lim_{x\to -\infty} h(x) = \frac{1}{\kappa}.$$ 

\item[(ii)] Suppose that $X$ is recurrent, then either
\begin{equation*}
\lim_{x\rightarrow \infty} h(x) = \infty \quad \text{or} \quad \lim_{x\to -\infty} h(x) = \infty.
\end{equation*} 
\end{enumerate}
\end{lemma}

\begin{remark}
The case where $X$ is transient and $\mathbb{E}(X_1^+) = \mathbb{E}(X_1^-)=\infty$ is not covered by the latter Lemma, but could be analysed using (\ref{htransient}).
\end{remark}

\begin{proof}[Proof of Lemma \ref{limitbehaviourofh}]
We only prove (i) in the case $0<\mu\leq \infty$, the other case can be proved similarly. Set $f(x)=u_1(x)$, $x\in\mathbb{R}$. Note that $u_0(x) = \sum_{n=1}^\infty f^{\ast n}(x)$. Indeed, 
\begin{eqnarray*}
\sum_{n=1}^\infty f^{\ast n}(x) dx &=& \sum_{n=1}^\infty \int_0^\infty  \frac{s^{n-1}}{(n-1)!} e^{-s} \mathbb{P}(X_s\in dx)ds \\
&=& \int_0^\infty e^{-s} \sum_{n=1}^\infty  \frac{s^{n-1}}{(n-1)!}  \mathbb{P}(X_s\in dx)ds \\
&=& \int_0^\infty  \mathbb{P}(X_s\in dx)ds \\
&=& u_0(x)dx.
\end{eqnarray*}
Furthermore, the Fourier transform of $f$ is given by $\widehat{f}(\lambda) = 1/(1+\psi(\lambda))$, $\lambda\in \mathbb{R}$. Since $h(x) = u_0(0)-u_0(-x)$, it suffices to compute the limit at infinity of $\sum_{n=1}^\infty f^{\ast n}(x)$. To that aim, we use the main result in \cite{Smith55}, which states that if
\begin{enumerate}
\item[(a)] $\lim_{|x|\to \infty} f(x) = 0$,

\item[(b)] $f$ is in $L_{1+\epsilon}$, for some $\epsilon>0$,
\end{enumerate}
then 
  $$\sum_{n=1}^\infty f^{\ast n}(x) \to \frac{1}{\mu}, \text{ as } x\to \infty, \quad \sum_{n=1}^\infty f^{\ast n}(x) \to 0, \text{ as } x\to -\infty.$$
The condition (a) is obtained from the Riemann-Lebesgue theorem. To show that (b) is satisfied we use Plancherel's theorem (see \cite[p. 186]{Rudin87}, \cite[p. 202]{Widder46}). Plancherel's theorem allow us to ensure that $f\in L_2$ if and only if $\widehat{f}\in L_2$. This fact can be easily obtained from Theorem~\ref{theoKesten}, since
$$\int_{\mathbb{R}}|\widehat{f}(\lambda)|^2d\lambda=\int_{\mathbb{R}}\frac{1}{|1+\psi(\lambda)|^{2}}d\lambda\leq \int_{\mathbb{R}}\frac{Re(1+\psi(\lambda))}{|1+\psi(\lambda)|^{2}}d\lambda=\int_{\mathbb{R}}Re\left(\frac{1}{1+\psi(\lambda)}\right)d\lambda<\infty.$$
This concludes the first part of the lemma.

To prove the second part of the Lemma, we consider the function $h^\ast$ which is defined in Section \ref{auxfunction} below. There, it is shown that $h^\ast(x) = h(x) + h(-x) =\mathbb{E}(L_{T_{x}})$ (see (\ref{equ3333}) and (\ref{equ3434}) for details). We will prove that $h^\ast(x)$ tends to infinity as $x\to\infty$ when $X$ is recurrent and thus obtain (ii). The proof is as follows. Let $\mathbf{e}_q$ be an independent exponential time. Observe the elementary inequality
 $$h^\ast(x)=\mathbb{E}(L_{T_{x}})\geq \mathbb{E}(L_{T_{x}}1_{\{T_{x}\geq \mathbf{e}_q\}})\geq \mathbb{E}(L_{\mathbf{e}_q}1_{\{T_{x}\geq\mathbf{e}_q\}}),$$ 
Hence, by Fatou's lemma
$$\liminf_{x\to\infty}h^\ast(x)\geq \mathbb{E}(L_{\mathbf{e}_q}\liminf_{x\to\infty}1_{\{T_{x}\geq \mathbf{e}_q\}}).$$
Now, by the Riemann-Lebesgue theorem 
$$\lim_{x\to\infty}\mathbb{E}(e^{-qT_{x}})=\lim_{x\to\infty}\frac{u_{q}(x)}{u_{q}(0)}=0,\quad \forall\, q>0.$$ 
Hence, it follows that $T_{x}$ converges weakly towards $\infty$ as $x\to \infty$. On the other hand, applying the Lemma 3 in \cite[Chapter V]{Bertoin},
  $$  u_q(0) = \mathbb{E}\left( \int_0^\infty e^{-qt} dL_t \right) =  \mathbb{E}\left( \int_0^\infty L_t qe^{-qt} dt \right) = \mathbb{E}(L_{\mathbf{e}_q}),\quad \forall\,q>0.$$
We have so proved that
  $$\liminf_{x\to\infty}h^\ast(x)\geq \mathbb{E}(L_{\mathbf{e}_q})  = u_q(0),\quad \forall\,q>0.$$ 
Letting $q$ tend to $0$ and using that in the recurrent case $u_{0}(0)= \infty$, we obtain 
  $$h^{*}(x) = h(x) + h(-x) \to \infty, \quad x\to \infty.$$ 
\end{proof}

\subsection{An auxiliary function} \label{auxfunction}

Let $(h_q^\ast)_{q>0}$ be the increasing sequence of functions defined by
\begin{equation*}
h_q^\ast(x) = \mathbb{E}\left(\int_0^{T_x} e^{-qt} dL_t\right), \quad q>0, x\in\mathbb{R}.
\end{equation*}
The sequence $(h_q^\ast)_{q>0}$ satisfies the following properties.

\begin{proposition}
For any $q>0$, the function $h_q^\ast$ is symmetric, nonnegative, subadditive continuous, which can be expressed in terms of the $q$-resolvent density as
\begin{equation} \label{equ2828}
h_q^\ast(x) = u_q(0) - \frac{u_q(x)u_q(-x)}{u_q(0)}, \quad x\in\mathbb{R}.
\end{equation}
\end{proposition}
\begin{proof}
By definition, $h_q^\ast$ is a non negative function. The continuity and symmetry of $h_q^\ast$ is obtained from (\ref{equ2828}). Thus, it only remains to prove (\ref{equ2828}) and that $h_q^\ast$ is subadditive.

First, we recall an expression that establishes a relation between resolvent densities and local times, (see Lemma 3 and the commentary before Proposition 4 in \cite[Chapter V]{Bertoin}):
\begin{equation} \label{equ2929}
u_q(-x) = \mathbb{E}_x\left( \int_0^\infty e^{-qt} dL_t \right) = \mathbb{E}\left( \int_0^\infty e^{-qt} dL(x,t) \right), \quad q>0, x\in\mathbb{R},
\end{equation}
where $(L(x,t), t\geq 0)$ is the local time at point $x$ for $(X,\mathbb{P})$. Thus, using the latter expression, we have
\begin{equation} \label{equ3030}
u_q(0) = \mathbb{E}\left( \int_0^\infty e^{-qt} dL_t \right) = h_q^\ast(x) + \mathbb{E}\left( \int_{T_x}^\infty e^{-qt} dL_t, T_x<\infty \right).
\end{equation}
On the other hand, by Markov and additivity properties of the local time, it follows 
\begin{eqnarray*}
\mathbb{E}\left( \int_{T_x}^\infty e^{-qt} dL_t, T_x<\infty \right) &=& \mathbb{E}\left( e^{-qT_x} \int_0^{\infty} e^{-qu} dL_{u+T_x}, T_x<\infty \right) \\
&=& \mathbb{E}(e^{-qT_x}, T_x<\infty) \mathbb{E}_x\left( \int_0^\infty e^{-qu} dL_u \right) \\
&=& \widehat{\mathbb{E}}_x(e^{-qT_0}, T_0<\infty) \mathbb{E}_x\left( \int_0^\infty e^{-qu} dL_u \right). \\
\end{eqnarray*}
Then, using (\ref{equ999}) and (\ref{equ2929}), the equation (\ref{equ3030}) becomes 
  $$u_q(0) = h_q^\ast(x) + \frac{u_q(x)}{u_q(0)} u_q(-x), \quad x\in\mathbb{R}.$$
From where (\ref{equ2828}) follows.

Now, we prove the subadditivity of $h_q^\ast$. The procedure is similar to the one used to prove the subadditivity of $h_q$ in Lemma \ref{propertiesofhq}. We repeat the arguments for clarity. First, by (\ref{equ777}) and (\ref{equ999}) we can write (\ref{equ2828}) as
\begin{equation} \label{equ3131}
h_q^\ast(x) = u_q(0)(1-\mathbb{E}_x(e^{-qT_0})\widehat{\mathbb{E}}_x(e^{-qT_0})).
\end{equation}
Since, for any $x\in\mathbb{R}$, $\mathbb{E}_x(e^{-qT_0}), \widehat{\mathbb{E}}_x(e^{-qT_0})\leq 1$, it follows
\begin{equation*}
(1-\mathbb{E}_x(e^{-qT_0})\widehat{\mathbb{E}}_x(e^{-qT_0}))(1-\mathbb{E}_y(e^{-qT_0})\widehat{\mathbb{E}}_y(e^{-qT_0}))\geq 0, \quad x,y\in\mathbb{R}.
\end{equation*}
The latter relation and (\ref{equ1717}) imply 
\begin{align*}
1-\mathbb{E}_x(e^{-qT_0})\widehat{\mathbb{E}}_x(e^{-qT_0}) + 1-\mathbb{E}_y(e^{-qT_0})\widehat{\mathbb{E}}_y(e^{-qT_0}) &\geq 1 - \mathbb{E}_x(e^{-qT_0})\mathbb{E}_y(e^{-qT_0})\widehat{\mathbb{E}}_x(e^{-qT_0})\widehat{\mathbb{E}}_y(e^{-qT_0}) \\
&\geq 1-\mathbb{E}_{x+y}(e^{-qT_0})\widehat{\mathbb{E}}_{x+y}(e^{-qT_0}), 
\end{align*}
for all $x,y\in\mathbb{R}$. Hence, by (\ref{equ3131})
\begin{equation*}
h_q^\ast(x) + h_q^\ast(y) \geq h_q^\ast(x+y), \quad x,y\in\mathbb{R}.
\end{equation*}
This ends the proof.
\end{proof}

\begin{remark}
With the help of the expression (\ref{equ2828}), $h_q^\ast$ can be written in terms of the function $h_q$ as: 
\begin{equation} \label{equ3232}
h_q^\ast(x) = h_q(x)+h_q(-x)-\frac{1}{u_q(0)}h_q(x)h_q(-x), \quad x\in\mathbb{R}.
\end{equation}
\end{remark}

Now, define $h^\ast$ by
\begin{equation*}
h^\ast(x) = \lim_{q\rightarrow 0} h_q^\ast(x), \quad x\in\mathbb{R}.
\end{equation*}
Since $h$ is finite, then (\ref{equ3232}) implies that $h^\ast(x)$ is finite for all $x\in\mathbb{R}$. Furthermore, since 
\begin{equation*}
h_q^\ast(x) = \mathbb{E}\left(\int_0^{T_x} e^{-qt} dL_t\right), 
\end{equation*}
then
\begin{equation} \label{equ3333} 
h^\ast(x) = \mathbb{E}(L_{T_x}). 
\end{equation}
It is known that $L_{T_x}$ is the first time where an excursion from zero that hits $x$ appears. Thus, $h^\ast(x)$ is the expected value of an exponential random variable. We also note that by (\ref{equ3232}), in the recurrent symmetric case, $h^\ast$ corresponds to $2h^Y$, where $h^Y$ is the invariant function given in \cite{Yano10}.

To state the following lemma let us introduce further notation that will be mainly used in the next two proofs. Let $Z=(Z_{t}, t\geq 0)$ be the process given by
  \[
  Z_t = X_{t}-Y_{t}, \quad t\ge 0,
  \] 
where $X=(X_{t}, t\geq 0)$ and $Y=(Y_{t}, t\geq 0)$ are independent and identically distributed L\'evy processes both with characteristic exponent $\psi$ under $\mathbb{P}$. In other words, $(Z_{t}, t\geq 0)$ is the symmetric L\'evy process obtained by symmetrisation of $X$. A straightforward computation shows that the characteristic exponent of $Z$ is $2Re(\psi).$ The objects associated to $Z$ will be denoted with a tilde, so $\widetilde{u_{q}}$ will denote its $q$-resolvent density and $\widetilde{h}_{q}(x)=\widetilde{u}_{q}(0)-\widetilde{u}_{q}(-x),$ for $q>0$, $x\in \mathbb{R}.$

\begin{lemma} \label{lemmaap1}
\begin{enumerate}
\item[(i)] For any $x\in \mathbb{R}$, $\lim_{q\rightarrow 0} qu_q(x) = 0$.

\item[(ii)] For any $q,r>0$, $x\in\mathbb{R}$, we have the following identity
\begin{equation*}
\int_\mathbb{R} u_q(y-x)u_r(y) dy = \int_{\mathbb{R}} \widetilde{u}_{q+r}(y)[ u_r(x-y) + u_q(y-x)] dy. 
\end{equation*}
\end{enumerate}
\end{lemma}
\begin{proof}
Recall the identity
  $$\frac{u_q(x)}{u_q(0)} = \widehat{\mathbb{E}}_{x}(e^{-qT_0}), \quad x\in\mathbb{R}.$$
Hence, $qu_q(x) \sim \widehat{\mathbb{P}}_{x}(T_0<\infty)qu_q(0)$ as $q\downarrow 0$. Thus, it suffices to prove the case $x=0$. Thanks to (\ref{equ1919}), we have
\begin{equation*}
qu_q(0) = \frac{1}{2\pi} \int_{\mathbb{R}} Re \left( \frac{q}{q+\psi(\lambda)} \right) d\lambda.
\end{equation*} 
For every $q>0$, let $j_q$ be the integrand in the above display. Now, we observe the following
\begin{align*}
j_q(\lambda) &= Re\left( \frac{q}{q+\psi(\lambda)}\right) \\
&= \frac{ q(q+Re(\psi(\lambda))) }{ [q+Re(\psi(\lambda))]^2 + [Im(\psi(\lambda))]^2} \\
&=  \left[1 + \displaystyle{ \frac{Re\psi(\lambda)}{q} + \frac{(Im\psi(\lambda))^2}{q(q+Re\psi(\lambda))} } \right]^{-1}, \quad q>0, \lambda\in \mathbb{R}.
\end{align*}
Hence, $j_q\downarrow 0$, as $q\downarrow 0$. Thus, $0\leq j_q(\lambda) \leq j_1(\lambda)$, $0<q<1$, $\lambda \in \mathbb{R}$ and $j_1$ is integrable by Theorem \ref{theoKesten}. Then, by the dominated convergence theorem $\lim_{q\rightarrow 0} qu_q(0) = 0$. This shows (i).

Now, let $f$ be a positive, bounded, measurable function. By construction we have
\begin{align*}
\int_\mathbb{R} \int_\mathbb{R} u_q(y-x)u_r(y) f(x) dy dx &= \int_\mathbb{R} \int_\mathbb{R} f(y-z)u_q(z) u_r(y) dy dz \\
&= \frac{1}{rq} \mathbb{E} \left( f(X_{\mathbf{e}_r} - Y_{\mathbf{e}_q}) \right) \\
&= \frac{1}{rq} \mathbb{E} \left( f(X_{\mathbf{e}_r} - Y_{\mathbf{e}_q}) \mathbf{1}_{\{\mathbf{e}_r > \mathbf{e}_q \}} \right) \\
&\quad + \frac{1}{rq} \mathbb{E} \left( f(-(Y_{\mathbf{e}_q} - X_{\mathbf{e}_r}) ) \mathbf{1}_{\{\mathbf{e}_q > \mathbf{e}_r \}} \right),
\end{align*}
where $\mathbf{e}_q$, $\mathbf{e}_r$ are exponential random variables with parameters $q>0$ and $r>0$, respectively, such that $\mathbf{e}_q$, $\mathbf{e}_r,$ $Y$ and $X$ are independent under $\mathbb{P}$. Applying the independence and stationarity property of the increments of $X$, we observe that the first term in the latter equation becomes
\begin{align*}
\frac{1}{rq} \mathbb{E} \left( f(X_{\mathbf{e}_r} - Y_{\mathbf{e}_q}) \mathbf{1}_{\{\mathbf{e}_r > \mathbf{e}_q \}} \right) &= \frac{1}{rq} \mathbb{E} \left( f(X_{\mathbf{e}_r} - X_{\mathbf{e}_q} + Z_{\mathbf{e}_q}) \mathbf{1}_{\{\mathbf{e}_r > \mathbf{e}_q \}} \right) \\
&= \int_0^\infty \int_s^\infty  \mathbb{E}(f(X_t-X_s + Z_s)) e^{-rt}e^{-qs}dt ds \\
&= \int_0^\infty \int_s^\infty  \mathbb{E}(f(W_{t-s}+Z_s))e^{-r(t-s)} e^{-(r+q)s}dtds,
\end{align*}
where $W=(W_t, t\geq 0)$ has the same law as $X$ and is independent of $Z$ under $\mathbb{P}$. Thus, 
\begin{align*}
\frac{1}{rq} \mathbb{E} \left( f(X_{\mathbf{e}_r} - Y_{\mathbf{e}_q}) \mathbf{1}_{\{\mathbf{e}_r > \mathbf{e}_q \}} \right) &= \int_0^\infty \int_s^\infty e^{-r(t-s)} \mathbb{E}(f(W_{t-s}+Z_s)) e^{-(r+q)s}dtds \\
&= \int_0^\infty \int_0^\infty e^{-rt} \mathbb{E}(f(W_t + Z_s)) e^{-(q+r) s} dt ds \\
&= \int_\mathbb{R} \int_\mathbb{R} f(z+y) u_r(z) \widetilde{u}_{q+r}(y) dy dz \\ 
&= \int_\mathbb{R} \int_\mathbb{R} \widetilde{u}_{q+r}(y) u_r(x-y) f(x) dy dx.
\end{align*}
In the same way, it can be verified that
\begin{align*}
\frac{1}{rq} \mathbb{E} \left( f(-(Y_{\mathbf{e}_q} - X_{\mathbf{e}_r}) ) \mathbf{1}_{\{\mathbf{e}_q > \mathbf{e}_r \}} \right) &= \frac{1}{rq} \mathbb{E} \left( f(-(Y_{\mathbf{e}_q} - Y_{\mathbf{e}_r}) + Z_{\mathbf{e}_r} ) \mathbf{1}_{\{\mathbf{e}_q > \mathbf{e}_r \}} \right) \\
&= \int_\mathbb{R} \int_\mathbb{R} \widetilde{u}_{q+r}(y) u_q(y-x) f(x) dy dx.
\end{align*}
Thus, we have
  $$\int_\mathbb{R} \int_\mathbb{R} u_q(y-x)u_r(y) f(x) dy dx = \int_\mathbb{R} \int_\mathbb{R} \widetilde{u}_{q+r}(y) \left[ u_r(x-y) +  u_q(y-x) \right] f(x) dy dx,$$
for all positive, bounded, measurable function $f$. By the continuity and boundeness of $u_r$ and $u_q$, we conclude that
  $$\int_\mathbb{R} u_q(y-x)u_r(y) dy =  \int_{\mathbb{R}} \widetilde{u}_{q+r}(y)[ u_r(x-y) + u_q(y-x)] dy, $$
for any $q,r>0$, $x\in\mathbb{R}$. 
\end{proof}

Some properties of the function $h^\ast$ are summarized in the following lemma.

\begin{lemma} \label{excessiveofhast1}
The function $h^\ast$ is a symmetric, nonnegative, subadditive, continuous function which vanishes only at the point $x=0$ and $\lim_{|x|\rightarrow \infty} h^\ast(x)=\kappa^{-1}$. Furthermore, $h^\ast$ is integrable with respect to semigroup of the process killed at $T_0$, i.e., $P_t^0 h^\ast(x)<\infty$, for all $t>0$, $x\in\mathbb{R}$. 
\end{lemma}
\begin{proof}
From the definition of $h_q^\ast$ and (\ref{equ2828}) the non negativity and symmetry of $h^\ast$ follows. The subadditivity of $h^\ast$ is obtained from subadditivity of the sequence $(h_q^\ast)_{q>0}$. We observe that from (\ref{equ3232}), we can write $h^\ast$ in terms of $h$ as 
\begin{equation} \label{equ3434} 
h^\ast(x) = h(x)+h(-x)-\kappa h(x)h(-x),
\end{equation}
where $\kappa = \lim_{q\rightarrow 0} \frac{1}{u_q(0)}$. Hence, $h^\ast$ is continuous.

Now, we prove that $\lim_{|x|\rightarrow \infty} h^\ast(x) = \kappa^{-1}$. In Lemma \ref{limitbehaviourofh} it has been proved that if $X$ is recurrent then $\lim_{x\to \infty} h^\ast(x) = \infty$. Since $h^\ast$ is a symmetric function, the same limit is obtained as $x\to -\infty$. Then, $\lim_{|x|\rightarrow \infty} h^\ast(x) = \kappa^{-1}$ when $X$ is recurrent. Suppose that $X$ is transient. In Lemma \ref{limitbehaviourofh} we established that
  $$\liminf_{x\to \infty} h^\ast(x) \geq u_q(0), \quad \forall \, q>0,$$
without further assumptions. Hence, letting $q$ tends to $0$, we obtain, $\liminf_{x\to \infty} h^\ast(x) \geq \kappa^{-1}$. On the other hand, taking $\mathbf{e}_q$ an exponential independent time with parameter $q$, we have
\begin{equation*}
\begin{split}
\mathbb{E}(L_{T_{x}})&=\mathbb{E}(L_{T_{x}}1_{\{T_{x}\geq \mathbf{e}_q\}}) + \mathbb{E}(L_{T_{x}}1_{\{T_{x}< \mathbf{e}_q\}})\\
&\leq \mathbb{E}\left(L_{T_{x}}\left(1-e^{-q T_{x}}\right)\right)  + \mathbb{E}(L_{\mathbf{e}_q})\\
&=\mathbb{E}\left(L_{T_{x}}\left(1-e^{-q T_{x}}\right)\right)+u_q(0), \quad \forall \, x\in \mathbb{R}.
\end{split}
\end{equation*}
Then, since $X$ is transient, $\mathbb{E}(L_{T_x})<\infty$. Thus, by the dominated convergence theorem 
  $$h^\ast(x) = \mathbb{E}(L_{T_x}) \leq \kappa^{-1}.$$
Hence, $\limsup_{x\to \infty} h^\ast(x)\leq \kappa^{-1}$. Therefore, $\lim_{x\to \infty} h^\ast(x) = \kappa^{-1}$ if $X$ is transient. Since $h^\ast(x)$ is a symmetric function the same limit is obtained as $x\to -\infty$. Therefore, $\lim_{|x|\rightarrow \infty} h^\ast(x) = \kappa^{-1}$.

From the defintion, $h^\ast(0)=0$. To prove that $x=0$ is the only point where $h^\ast$ vanishes, we proceed by contradiction. Suppose that $h^\ast(x_0)=0$, for some $x_0\neq 0$. Using the subadditivity and symmetry of $h^\ast$, and making induction we get that $h^\ast(kx_0)=0$ for all $k\in\mathbb{Z}.$ Since $\lim_{|x|\to\infty}h^\ast(x) = \kappa^{-1}>0$, the claim $h^\ast(kx_0)=0$, for all $k\in\mathbb{Z}$ is a contradiction. Therefore, $h^\ast(x)>0$, for all $x\neq 0$.

Finally, we prove that $h^\ast$ is $P_t^0$-integrable. For $x\in\mathbb{R}$, we write $\widehat{h}_q(x)=h_q(-x)$, $q>0$, and $\widehat{h}(x) = \lim_{q\to 0} \widehat{h}_q(x)$. Let $S$ be the function defined by $S(x) = h(x) + \widehat{h}(x)$. By (\ref{equ3434}), $h^\ast(x) \leq S(x)$, $x\in\mathbb{R}$. Thus, it suffices to show that $S$ is $P_t^0$-integrable. Now, by (\ref{equ666}), the following identities hold for $0<r<q$ ,
\begin{eqnarray} \nonumber
U_qh_r(x) &=& \int_{\mathbb{R}} u_q(y-x)h_r(y)dy \\ \nonumber
&=&
\int_{\mathbb{R}} u_q(y-x)\{u_r(0)-u_r(-y)\}dy \\ \nonumber
&=&
\frac{u_r(0)}{q}-\int_{\mathbb{R}} u_q(y-x)u_r(-y)dy \\ \nonumber
&=&
\frac{u_r(0)}{q} - \frac{1}{q-r}\{ u_r(-x)-u_q(-x) \} \\ \label{eqaux3-2}
&=&
\frac{h_r(x)}{q} - \frac{ ru_r(-x) }{ q(q-r) } + \frac{ u_q(-x) }{ q-r }.
\end{eqnarray}
Thanks to Lemma \ref{lemmaap1} (i), $h_r(x)\rightarrow h(x)$ as $r\rightarrow 0$ and Fatou's lemma, we obtain 
\begin{equation} \label{eqaux5-2}
U_{q}h(x) \leq \frac{h(x)+u_q(-x)}{q}, \quad q>0, x\in\mathbb{R}.
\end{equation} 
Now, we will obtain a bound for $U_q\widehat{h}$ in terms of the process $Z$, the symmetrisation of $X$, and the sequence $(\widetilde{h}_{r})$ introduced before Lemma \ref{lemmaap1}. We note that since the characteristic exponent of $Z$ is $2Re(\psi)$, then by (\ref{equ111}) 
  \[
  \widetilde{h}_r(x) = \widetilde{u}_r(0) - \widetilde{u}_r(-x) = \frac{1}{2\pi} \int_{\mathbb{R}}\frac{1-\cos(\lambda x)}{r+2Re(\psi(\lambda))} d\lambda, \quad r>0, x\in\mathbb{R}.
  \]
Furthermore, by monotone convergence theorem and Remark \ref{intcond}
  \[
  \lim_{r\to 0} \widetilde{h}_r(x) = \frac{1}{2\pi} \int_{\mathbb{R}}\frac{1-\cos(\lambda x)}{2Re(\psi(\lambda))} d\lambda \leq \frac{1}{2\pi} \int_{\mathbb{R}}\frac{1\wedge (\lambda x)^2}{Re(\psi(\lambda))} d\lambda < \infty, \quad x\in\mathbb{R}.
  \]
On the other hand, for all $r>0$, $x\in \mathbb{R}$, it holds
\begin{equation*}
\begin{split}
u_r(0) - u_r(x) &\leq h_{r}(x)+h_{r}(-x) =  \frac{1}{2\pi} \int_{\mathbb{R}}\frac{1-\cos(\lambda x)}{|r+\psi(\lambda)|^2}[r+Re(\psi(\lambda))]d\lambda\\
&\leq  \frac{1}{2\pi} \int_{\mathbb{R}}\frac{1-\cos(\lambda x)}{r+Re(\psi(\lambda))}d\lambda = 2\left( \frac{1}{2\pi} \int_{\mathbb{R}}\frac{1-\cos(\lambda x)}{2r+2Re(\psi(\lambda))}d\lambda \right) = 2\widetilde{h}_{2r}(x). 
\end{split}
\end{equation*}
Proceeding as in (\ref{eqaux3-2}) and with the help of Lemma \ref{lemmaap1} (ii), we obtain
\begin{eqnarray*} \nonumber
U_q\widehat{h}_r(x) &=& \frac{1}{q}u_{r}(0)-U_{q}u_{r}(x) = \frac{1}{q}u_{r}(0)-\int_\mathbb{R} u_{q}(y-x) u_{r}(y) dy \\
&=& \int_\mathbb{R} \widetilde{u}_{q+r}(y) \left[u_{r}(0)-u_{r}(x-y)\right]dy +\frac{r}{q} u_{r}(0) \int_\mathbb{R} \widetilde{u}_{q+r}(y)dy - \int_\mathbb{R} \widetilde{u}_{q+r}(y)u_{q}(y-x) dy\\
&=& \int_\mathbb{R} \widetilde{u}_{q+r}(y) \left[u_{r}(0)-u_{r}(x-y)\right] dy + \frac{r}{q(q+r)}u_{r}(0) -\frac{1}{q+r} \mathbb{E}(u_{q}(Z_{\mathbf{e}_{q+r}}-x)) \\
&\leq& 2\int_\mathbb{R} \widetilde{u}_{q+r}(y)\widetilde{h}_{2r}(x-y) dy +\frac{r}{q(q+r)}u_{r}(0)-\frac{1}{q+r}\mathbb{E}(u_{q}(Z_{\mathbf{e}_{q+r}}-x)) \\
&=& 2\left[\frac{\widetilde{h}_{2r}(x)}{q+r} - \frac{ (2r) \widetilde{u}_{2r}(-x) }{ (q+r)(q-r) } + \frac{ \widetilde{u}_{q+r}(-x) }{ q-r }\right] + \frac{r}{q(q+r)}u_{r}(0) -\frac{1}{q+r}\mathbb{E}(u_{q}(Z_{\mathbf{e}_{q+r}}-x)),
\end{eqnarray*}
notice that in the final identity we applied the identity (\ref{eqaux3-2})  for the process $Z$. For the last term in the above equality, we have that since 
  \[
  \widetilde{u}_{q+r}(x) = \frac{1}{2\pi} \int_\mathbb{R} \frac{ \cos(\lambda x) }{ q+r + 2Re(\psi(\lambda))} d\lambda,
  \]
it follows $\widetilde{u}_{q+r}(x) \uparrow \widetilde{u}_{q}(x)$, as $r\to 0$, for any $x\in \mathbb{R}$. Hence by the monotone convergence theorem
\begin{equation*}
\begin{split}
\lim_{r\to 0} \frac{1}{q+r}\mathbb{E}(u_{q}(Z_{\mathbf{e}_{q+r}}-x)) &= \lim_{r\to 0} \int_\mathbb{R} \widetilde{u}_{q+r}(y)u_{q}(y-x) dy \\
&=  \int_\mathbb{R} \widetilde{u}_{q}(y)u_{q}(y-x) dy \\
&=  \frac{1}{q}\mathbb{E}(u_{q}(Z_{\mathbf{e}_{q}}-x)),
\end{split}
\end{equation*}
which is finite since $u_q$ is a bounded function. Thus, using again Lemma \ref{lemmaap1} (i), $\widehat{h}_r(x) \to \widehat{h}(x)$, $\widetilde{h}_r(x) \to \widetilde{h}(x)$ as $r\to 0$, and Fatou's lemma, it follows that
\begin{equation} \label{eqaux6-2}
U_q\widehat{h}(x) \leq 2\left[\frac{\widetilde{h}(x)}{q} + \frac{ \widetilde{u}_{q}(-x) }{ q }\right]-\frac{1}{q}\mathbb{E}(u_{q}(Z_{\mathbf{e}_{q}}-x))
\end{equation}
is finite for all $q>0$, $x\in\mathbb{R}$. Adding (\ref{eqaux5-2}) and (\ref{eqaux6-2}), we obtain that for any $q>0$, $x\in\mathbb{R}$,
\begin{equation} \label{eqaux7-2}
qU_qS(x) \leq h(x)+u_q(-x) + 2[\widetilde{h}(x) + \widetilde{u}_{q}(-x)] - \mathbb{E}(u_{q}(Z_{\mathbf{e}_{q}}-x))
\end{equation}
is finite. Hence, the function $S$ is $P_t$-integrable and therefore $P_t^0$-integrable. 
\end{proof}

\begin{remark} \label{remarkofh}
By (\ref{equ1616}) and (\ref{equ3131}), we have $h_q(x)\leq h_q^\ast(x)\leq h^\ast(x)\leq S(x)$, for all $q>0$, $x\in\mathbb{R}$. On the other hand, the Lemma \ref{excessiveofhast1} and its proof ensure that $h^\ast$ satisfies that $P_t^0h^\ast(x)$ is finite for all $t>0$, $x\in\mathbb{R}$ and
\begin{equation} \label{eqaux8-2}
qU_qh^\ast(x) \leq \alpha_q(x), \quad q>0, x\in \mathbb{R},
\end{equation}
where $\alpha_q(x)$ is the function in the right hand side of (\ref{eqaux7-2}). The latter inequality will be useful in the proofs of Lemma \ref{resolventofh} and assertion (iii) in Theorem \ref{propofh}. 
\end{remark}

The next lemma ensures that the inequality obtained in (\ref{eqaux5-2}) in fact is an equality. This result was established in \cite{Yano10} in the symmetric case.

\begin{lemma} \label{resolventofh}
For any $q>0$, $x\in\mathbb{R}$,  
\begin{equation*}
U_qh(x) = \frac{h(x)+u_q(-x)}{q}. 
\end{equation*}
\end{lemma}

\begin{proof}
Remark \ref{remarkofh} states that the function $h^\ast$ satisfies $h_q(x)\leq h^\ast(x)$, $U_qh^\ast(x)<\infty$, for all $q>0$, $x\in\mathbb{R}$. Then, by the dominated convergence theorem and (\ref{eqaux3-2}), it follows
\begin{equation*}
U_qh(x) = \lim_{r\rightarrow 0} U_q h_r(x) = \frac{h(x)+u_q(-x)}{q}.
\end{equation*}
\end{proof}

\section{Proofs of the main results} \label{proofsofmain}

We will first prove Theorem~\ref{propofh}. Recall that in Section~\ref{spofh} we proved the claims in (i) and (ii). Here we will prove (iii).

\begin{proof}[Proof of (iii) in Theorem \ref{propofh}]
The first part of the proof is inspired by the proof of Lemma 1 in \cite{Chaumont-Doney05}. Let $\mathbf{e}_q$ be an exponential random variable with parameter $q>0$, and independent of $(X,\mathbb{P})$. We claim that for $q>0$, $x\in\mathbb{R}$, 
\begin{equation} \label{equ3535}
\mathbb{E}_x(\mathbb{P}_{X_t}(T_0>\mathbf{e}_q)\mathbf{1}_{\{T_0>t\}}) = e^{qt}\left( \mathbb{P}_x(T_0>\mathbf{e}_q) - \int_0^t\mathbb{P}_x(T_0>s)qe^{-qs}ds \right).
\end{equation}
Indeed, by (\ref{equ1818}), we have
\begin{equation*} 
\mathbb{E}_x\left( \mathbb{P}_{X_t}(T_0>\mathbf{e}_q) \mathbf{1}_{\{t<T_0\}} \right) = \mathbb{P}_x(T_0>t+\mathbf{e}_q).
\end{equation*}
Now, making the change of variables $u=t+s$, we obtain 
\begin{align*}
\mathbb{P}_x(T_0>t+\mathbf{e}_q) &= \int_0^\infty \mathbb{P}_x(T_0>t+s)q^{-qs}ds \\
&= e^{qt} \int_t^\infty \mathbb{P}_x(T_0>u) qe^{-qu} du \\
&= e^{qt} \left( \int_0^\infty \mathbb{P}_x(T_0>u)qe^{-qu}du - \int_0^t\mathbb{P}_x(T_0>u)qe^{-qu}du \right) \\
&= e^{qt}\left( \mathbb{P}_x(T_0>\mathbf{e}_q) - \int_0^t\mathbb{P}_x(T_0>u)qe^{-qu}du \right).
\end{align*}
Hence, (\ref{equ3535}) follows.

By Remark \ref{remarkofh} and Lemma \ref{excessiveofhast1}, we derive that the sequence $(h_q)_{q>0}$ is dominated by $h^\ast$ and $h^\ast$ is integrable with respect to $P_t^0$ for any $t>0$. Then, using dominated convergence theorem, (\ref{equ1616}) and (\ref{equ3535}), it follows 
\begin{align*}
\mathbb{E}_x \left( h(X_t), t<T_0 \right) &= \mathbb{E}_x \left( \lim_{q\rightarrow 0} h_q(X_t), t<T_0 \right) \\
&= \lim_{q\rightarrow 0} \mathbb{E}_x \left( \frac{ \mathbb{P}_{X_t}(T_0>\mathbf{e}_q) }{ n(\zeta > \mathbf{e}_q) } \mathbf{1}_{\{t<T_0\}} \right) \\
&= \lim_{q\rightarrow 0} e^{qt} \left( \frac{ \mathbb{P}_x ( T_0> \mathbf{e}_q )}{ n(\zeta > \mathbf{e}_q) } - \int_0^t \frac{ \mathbb{P}_x ( T_0>u )}{n(\zeta > \mathbf{e}_q) } qe^{-qu}du \right) \\
&= h(x) - \frac{1}{n(\zeta)} \int_0^t \mathbb{P}_x ( T_0>u )du,
\end{align*}
where $n(\zeta) = \lim_{q\rightarrow 0} \int_0^\infty e^{-q t} n(\zeta>t)dt$. On the other hand, Lemma \ref{lemmaap1} (i) and (\ref{equ1212}) imply $n(\zeta)=\lim_{q\rightarrow 0}[qu_q(0)]^{-1} = \infty$. Therefore, we conclude
\begin{equation*}
\mathbb{E}_x \left( h(X_t), t<T_0 \right) = h(x), \quad t>0, x\in\mathbb{R}.
\end{equation*}

Now, we prove the second part of (iii) in Theorem \ref{propofh}. To that aim we compute the Laplace transform of $n(h(X_t), t<\zeta)$. From (\ref{equ1111}) and Lemma \ref{resolventofh}, we obtain that the Laplace transform of $n(h(X_t), t<\zeta)$ is given by
  $$\int_0^\infty e^{-qt}n(h(X_t), t<\zeta) dt = \int_{\mathbb{R}} h(x)\widehat{\mathbb{E}}_x[e^{-qT_0}] dx = \int_{\mathbb{R}} h(x)\frac{u_q(x)}{u_q(0)}dx = \frac{1}{u_q(0)} U_qh(0) = \frac{1}{q}.$$
Hence, the claim follows. 
\end{proof}

We will next prove Theorem~\ref{mainresult}.

\begin{proof}[Proof of Theorem \ref{mainresult}]
The only thing which has to be proved is the fact that $\mathbb{P}_0^\updownarrow$ is a Markovian probability measure with the same semigroup as under $\mathbb{P}_x^\updownarrow$, $x\in \mathcal{H}$ and that $\mathbb{P}_0^\updownarrow(X_0=0)=1$. Recall $n$ is a Markovian measure ($\sigma$-finite) with semigroup $(P_t^0, t\geq 0)$. Let $g$ be any bounded Borel function and $\Lambda \in \mathcal{F}_t$ and $t,s>0$:
\begin{eqnarray*}
\mathbb{E}_0^\updownarrow(\mathbf{1}_\Lambda g(X_{t+s})) &=& n( \mathbf{1}_\Lambda h(X_{t+s})g(X_{t+s}) \mathbf{1}_{\{t<\zeta\}}) \\
&=& n(\mathbf{1}_\Lambda \mathbb{E}_{X_t}^0(h(X_s)g(X_s)) \mathbf{1}_{\{t<\zeta\}}) \\
&=& n(\mathbf{1}_\Lambda h(X_t) \mathbb{E}_{X_t}^\updownarrow(g(X_s)) \mathbf{1}_{\{t<\zeta\}}) \\
&=& \mathbb{E}_0^\updownarrow(\mathbf{1}_\Lambda \mathbb{E}_{X_t}^\updownarrow(g(X_s))).
\end{eqnarray*}
This shows the first part. Now, we prove that $\mathbb{P}_0^\updownarrow(X_0=0)=1$. Since $X$ is right continuous at 0, it suffices to prove that for any $z>0$,
\begin{equation*}
\mathbb{P}_0^\updownarrow(|X_\epsilon|<z) \rightarrow 1,
\end{equation*}
as $\epsilon \rightarrow 0$. The latter is equivalent to prove
\begin{equation*}
\lim_{\epsilon \rightarrow 0} n(\mathbf{1}_{\{|X_\epsilon|>z\}} h(X_\epsilon) \mathbf{1}_{\{\epsilon<\zeta\}}) = 0.
\end{equation*}
Since for $s>0$, $n(h(X_s), s<\zeta) = 1$, the measure defined on $\mathcal{F}_s$, $\mathcal{Q}_s(\cdot):=n(\cdot,h(X_s), s<\zeta)$ is a probability measure. Then, from the Markov property, for all $\epsilon<s$, $\mathbb{P}_0^\updownarrow(|X_\epsilon|<z) = \mathcal{Q}_s(\mathbf{1}_{\{|X_\epsilon|<z\}})$. Using that excursions of the L\'evy process $(X,\mathbb{P})$ leave $0$ continuously, because 0 is assumed to be regular for itself, we derive $\mathbf{1}_{\{|X_\epsilon|<z\}} \rightarrow 1$, $\mathcal{Q}_s$-a.s. as $\epsilon\rightarrow 0$. The result follows from the dominated converge theorem.
\end{proof}

\begin{proof}[Proof of Theorem \ref{condavoidzero}] 
We proceed as in \cite{Chaumont-Doney05}. Let $x\in \mathcal{H}$, $\Lambda\in \mathcal{F}_t$, $t>0$. With the help of the Markov property and since $\mathbf{e}_q$ is independent of $(X, \mathbb{P})$, we can deduce the following
\begin{eqnarray*}
\mathbb{E}_x \left( \mathbf{1}_\Lambda \mathbf{1}_{\{t<\mathbf{e}_q\}} \mathbf{1}_{\{T_0>\mathbf{e}_q\}} \right)
&=& \int_0^\infty\mathbb{E}_x \left( \mathbf{1}_\Lambda \mathbf{1}_{\{t<T_0\}} \mathbf{1}_{\{T_0>s\}} \right)  \mathbf{1}_{\{t<s\}}qe^{-qs} ds \\
&=& \int_0^\infty\mathbb{E}_x \left( \mathbf{1}_\Lambda \mathbf{1}_{\{t<T_0\}}  \mathbb{E}_x \left( \mathbf{1}_{\{T_0>s\}} \circ \theta_t \mid \mathcal{F}_t \right) \right) \mathbf{1}_{\{t<s\}} qe^{-qs} ds \\
&=& \int_0^\infty \mathbb{E}_x \left( \mathbf{1}_\Lambda \mathbf{1}_{\{t<T_0\wedge s\}} \mathbb{P}_{X_t}(T_0>s) \right) qe^{-qs} ds \\
&=& \mathbb{E}_x \left( \mathbf{1}_\Lambda \mathbf{1}_{\{t<T_0\wedge \mathbf{e}_q\}} \mathbb{P}_{X_t}(T_0>\mathbf{e}_q) \right)\\
&=& n(\zeta>\mathbf{e}_q)\mathbb{E}_x \left( \mathbf{1}_\Lambda \mathbf{1}_{\{t<T_0\wedge \mathbf{e}_q\}} h_q(X_t)\right) \\
&=& \frac{1}{h_q(x)} \mathbb{E}_x \left( \mathbf{1}_\Lambda \mathbf{1}_{\{t<T_0\wedge \mathbf{e}_q\}} h_q(X_t)\right) \mathbb{P}_x(T_0>\mathbf{e}_q).
\end{eqnarray*}
The latter shows that for $\Lambda\in \mathcal{F}_t$,
\begin{equation} \label{equ3636}
\mathbb{P}_x(\Lambda, t<\mathbf{e}_q \mid T_0 > \mathbf{e}_q) = \frac{1}{h_q(x)} \mathbb{E}_x \left( \mathbf{1}_\Lambda h_q(X_t) \mathbf{1}_{\{t<T_0\wedge \mathbf{e}_q\}} \right).
\end{equation}
Now, recall that $h_q(x)\leq h_q^\ast(x) \leq h^\ast(x)$, $q>0$, $x\in\mathbb{R}$. Thus, 
\begin{equation*}
\mathbf{1}_{\{t<T_0\wedge \mathbf{e}_q\}} h_q(X_t) \leq \mathbf{1}_{\{t<T_0\}} h^\ast(X_t) \quad \mathrm{a.s.}
\end{equation*}
Furthermore, by Lemma \ref{excessiveofhast1}, $\mathbb{E}_x(h^\ast(X_t), t<T_0)$ is finite. Then, letting $q\rightarrow 0$, with the help of the dominated convergence theorem in (\ref{equ3636}), we obtain the desired result. 
\end{proof}

\begin{proof}[Proof of Theorem \ref{theo5}]
For every $s>0$, we consider $d_s=\inf\{u>s: X_u=0\}$, $g_s=\sup\{u\leq s: X_u=0\}$ and $G=\{g_u: g_u\neq d_u, u>0\}$. By definition, for every $q>0$, $\Lambda\in \mathcal{F}_t$, we have
\begin{eqnarray*}
\mathbb{P}^{\mathbf{e}_q} (\Lambda, t<\zeta) &=& \mathbb{E} (\mathbf{1}_\Lambda \circ k_{\mathbf{e}_q-g_{\mathbf{e}_q}} \circ \theta_{g_{\mathbf{e}_q}} \mathbf{1}_{\{t<\mathbf{e}_q-g_{\mathbf{e}_q}\}} ) \\
&=& \mathbb{E} \left( \int_0^\infty \mathbf{1}_\Lambda \circ k_{u-g_u} \circ \theta_{g_u} \mathbf{1}_{\{t<u-g_u\}} qe^{-qu}du \right) \\
&=& \mathbb{E} \left( \sum_{s\in G} e^{-qs} \int_s^{d_s} qe^{-q(u-s)} \mathbf{1}_\Lambda \circ k_{u-s} \circ \theta_s \mathbf{1}_{\{t<u-s\}} du \right).
\end{eqnarray*}
Now, using the compensation formula in excursion theory (see e.g. \cite{Bertoin}, \cite{Maisonneuve75}) and the Markov property of $n$, we obtain
\begin{align*}
\mathbb{E} \left( \sum_{s\in G} e^{-qs} \int_s^{d_s} qe^{-q(u-s)} \mathbf{1}_\Lambda \circ k_{u-s} \circ \theta_s \mathbf{1}_{\{t<u-s\}} du \right)
&=\mathbb{E}\left(\int_0^\infty e^{-qs}dL_s \right) n(\mathbf{1}_\Lambda \mathbf{1}_{\{t<\mathbf{e}_q<\zeta\}}) \\ 
&=  \mathbb{E}\left(\int_0^\infty e^{-qs}dL_s \right) n(\mathbf{1}_\Lambda \mathbb{P}_{X_t}(T_0>\mathbf{e}_q) \mathbf{1}_{\{t<\zeta\}}). 
\end{align*}
Using (\ref{equ1212}) and (\ref{equ3030}) we deduce
\begin{equation*}
\mathbb{E}\left(\int_0^\infty e^{-qs}dL_s \right) = u_q(0) = \frac{1}{n(\zeta>\mathbf{e}_q)}.
\end{equation*}
Thus, we obtain
\begin{equation} \label{equ3737}
\mathbb{P}^{\mathbf{e}_q} (\Lambda, t<\zeta) = \frac{ n(\mathbf{1}_\Lambda \mathbf{1}_{\{t<\mathbf{e}_q<\zeta\}}) }{ n(\zeta>\mathbf{e}_q) } = n( \mathbf{1}_\Lambda h_q(X_t) \mathbf{1}_{\{t<\zeta\}} ).
\end{equation}
First, we will prove that
  \[
  \lim_{q\to 0} \mathbb{P}^{\mathbf{e}_q} (t<\zeta) = 1.
  \]
The first equality of (\ref{equ3737}) can be written as follows
\begin{eqnarray*}
\mathbb{P}^{\mathbf{e}_q} (t<\zeta) &=& \frac{ n(t<\mathbf{e}_q<\zeta) }{ n(\zeta>\mathbf{e}_q) } \\
&=&  \frac{ n(\zeta>\mathbf{e}_q, \mathbf{e}_q > t) }{ n(\zeta>\mathbf{e}_q) } \\
&=&  \frac{ n(\zeta>\mathbf{e}_q) }{ n(\zeta>\mathbf{e}_q) } - \frac{ n(\zeta>\mathbf{e}_q, \mathbf{e}_q \leq t) }{ n(\zeta>\mathbf{e}_q) } \\
&=& 1 - \frac{ q }{ n(\zeta>\mathbf{e}_q) } n\left( \frac{ 1- e^{-q (t\wedge \zeta)}}{ q} \right).
\end{eqnarray*}
By (\ref{equ1212}) and Lemma \ref{lemmaap1} (i),
  \[
  \lim_{q\to 0} \frac{ q }{ n(\zeta>\mathbf{e}_q) } = \lim_{q\to 0} qu_q(0) = 0.
  \]
Using the inequality $1-e^{-x}\leq x$, for $x>0$ and since $n(t\wedge \zeta)<\infty$, the dominated convergence theorem implies
  \[ 
 \lim_{q\to 0} n\left( \frac{ 1- e^{-q (t\wedge \zeta)}}{ q} \right) = n(t\wedge \zeta) < \infty.
  \]
Thus, we conclude   
\[
  \lim_{q\to 0} \mathbb{P}^{\mathbf{e}_q} (t<\zeta) = 1-\lim_{q\to 0} \frac{ q }{ n(\zeta>\mathbf{e}_q) } n\left( \frac{ 1- e^{-q (t\wedge \zeta)}}{ q} \right) = 1.
  \]
Now, we prove that
  $$\lim_{q\rightarrow 0} \mathbb{P}^{\mathbf{e}_q} (\Lambda, t<\zeta) = n(\mathbf{1}_\Lambda h(X_t) \mathbf{1}_{\{t<\zeta\}}).$$
Fatou's lemma and (\ref{equ3737}) imply that for any $t>0$ and $\Lambda\in \mathcal{F}_t$
  \[
  \liminf_{q\to 0} \mathbb{P}^{\mathbf{e}_q}(\Lambda, t<\zeta) \geq n(\mathbf{1}_\Lambda h(X_t) \mathbf{1}_{\{t<\zeta\}}).
  \]
Furthermore, since $n(h(X_t), t<\zeta)= 1$ for all $t>0$ and $\mathbb{P}^{\mathbf{e}_q} (t<\zeta) \to 1$ as $q\to 0$, it follows that
\begin{eqnarray*}
\limsup_{q\to 0}  \mathbb{P}^{\mathbf{e}_q}(\Lambda, t<\zeta) &=& \limsup_{q\to 0}  \mathbb{P}^{\mathbf{e}_q}(t<\zeta) - \liminf_{q\to 0}  \mathbb{P}^{\mathbf{e}_q}(\Lambda^c, t<\zeta) \\
&=& 1- \liminf_{q\to 0}  \mathbb{P}^{\mathbf{e}_q}(\Lambda^c, t<\zeta) \\
&\leq& 1- n(\mathbf{1}_{\Lambda^c} h(X_t) \mathbf{1}_{\{t<\zeta\}}) \\
&=& n(\mathbf{1}_{\Lambda} h(X_t) \mathbf{1}_{\{t<\zeta\}}).
\end{eqnarray*}
Putting the pieces together we conclude that
  $$\lim_{q\rightarrow 0} \mathbb{P}^{\mathbf{e}_q} (\Lambda, t<\zeta) = n(\mathbf{1}_\Lambda h(X_t) \mathbf{1}_{\{t<\zeta\}}).$$
\end{proof}

We will now establish the claims in Proposition~\ref{transitorityofhprocess} and Proposition~\ref{relexcmeasure}. For that aim, let $U_q^\updownarrow$ be the $q$-resolvent for the process $X^\updownarrow=(X,\mathbb{P}_x^\updownarrow)_{x\in\mathcal{H}_0}$, with $U^\updownarrow = U_0^\updownarrow$. To prove that $X^\updownarrow$ is transient, we compute the density of $U^\updownarrow$. For $x,y\in\mathcal{H}$ and $q>0$, we have
\begin{equation} \label{equ3838}
u_q^\updownarrow(x,y) = \frac{h(y)}{h(x)} u_q^0(x,y).
\end{equation}
From (\ref{equ1111}) it can be deduced that for $y\in \mathcal{H}$, $q>0$,
\begin{eqnarray} \nonumber
u_q^\updownarrow(0,y)dy &=& \int_0^\infty e^{-qt} n(h(X_t)\mathbf{1}_{\{X_t\in dy\}}, t<\zeta) dt \\ \nonumber
&=& h(y)\widehat{\mathbb{E}}_y[e^{-qT_0}]dy \\ \label{equ3939}
&=& h(y)\frac{u_q(y)}{u_q(0)}dy.
\end{eqnarray}
Finally, by Theorem \ref{mainresult} (ii), $u_q^\updownarrow(x,0) = 0$, for all $x\in \mathcal{H}_0$. Thus, from the above equations, the density of $U^\updownarrow$ can be obtained. This is stated in the following lemma.

\begin{lemma} \label{densityofuaz}
Let $u_0^\updownarrow(x,y) = \lim_{q\rightarrow 0} u_q^\updownarrow(x,y)$, $x,y\in \mathcal{H}_0$. Then $u_0^\updownarrow(x,0) = 0$, for all $x\in \mathcal{H}_0$,
\begin{equation} \label{equ4040}
0\leq u_0^\updownarrow(x,y) = \frac{h(y)}{h(x)}[ h(x) + h(-y) - h(x-y) - \kappa h(x)h(-y)], \quad x\in \mathcal{H}, y\in \mathcal{H},
\end{equation}
and for $y\in \mathcal{H}$,
\begin{equation} \label{equ4141}
u_0^\updownarrow(0,y) = h(y)(1-\kappa h(-y)) = h^\ast(y)-h(-y).
\end{equation}
\end{lemma}
\begin{proof}
An easy computation gives
\begin{equation*}
\frac{u_q(-x)u_q(y)}{u_q(0)} = \frac{h_q(x)h_q(-y)}{u_q(0)} - h_q(x) - h_q(-y) + u_q(0), \quad  x\in \mathcal{H}, y\in \mathcal{H}.
\end{equation*}
Using this and (\ref{equ888}) it follows
\begin{equation} \label{equ4242}
u_q^0(x,y) = h_q(x) + h_q(-y) - h_q(x-y) - \frac{h_q(x)h_q(-y)}{u_q(0)}, \quad x\in \mathcal{H}, y\in \mathcal{H}.
\end{equation}
Letting $q\rightarrow 0$ in (\ref{equ3838}) and with help of (\ref{equ4242}), we obtain (\ref{equ4040}). The first equality in (\ref{equ4141}) is obtained from (\ref{equ3939}) recalling that for all $y$, $\lim_{q\rightarrow 0}[u_q(y)/u_q(0)]= \lim_{q\rightarrow 0}[1-(u_q(0))^{-1}h_q(-y)] = 1-\kappa h(-y)$. The second one follows from (\ref{equ3434}).
\end{proof}

\begin{remark} \label{remhast}
Note that from (\ref{equ3232}) and (\ref{equ4242}) we have $u_q^\updownarrow(x,x) = u_q^0(x,x) = h_q^\ast(x)$, $x\in \mathcal{H}$, which implies $u_0^\updownarrow(x,x)=h^\ast(x)$, $x\in \mathcal{H}$.  
\end{remark}

\begin{proof} [Proof of Proposition \ref{transitorityofhprocess}]
To obtain the transiency property of $X^\updownarrow$, we use Theorem 3.7.2 in \cite{Chung-Walsh05}, which states the following. If the conditions:
\begin{enumerate}
\item[(i)] $U^\updownarrow g$ is lower semi-continuous, for any non negative function $g$ with compact support; 

\item[(ii)] there exists a non negative function $f$ such that $0<U^\updownarrow f<\infty$ on $\mathcal{H}_0$;
\end{enumerate}
are satisfied, then the process $X^\updownarrow$ is transient.

Since $h$ is continuous, from Lemma \ref{densityofuaz} it follows $\lim_{x\rightarrow x'} u_0^\updownarrow(x,y) = u_0^\updownarrow(x',y)$, for all $y\in\mathcal{H}_0$. Let $g$ be a non negative function with compact support $K$. By Fatou's lemma, we have 
\begin{equation*}
\liminf_{x\rightarrow x'} \int_K g(y) u_0^\updownarrow(x,y) dy \geq \int_\mathbb{R} g(y)\liminf_{x\rightarrow x'} [u_0^\updownarrow(x,y)\mathbf{1}_K] dy 
= \int_K g(y)u_0^\updownarrow(x',y)dy.
\end{equation*}
This shows that for any $g$ non negative with compact support, the function
\begin{equation*}
x \longmapsto \int_\mathbb{R} g(y)u_0^\updownarrow(x,y) dy
\end{equation*}
is lower semi-continuous. Thus, condition (i) is satisfied.

Now, we will find a non negative function $f:\mathbb{R}\rightarrow \mathbb{R}^+$ such that $0<U^\updownarrow f(x)<\infty$. Let $f$ be given by
  $$f(y) = \left\{
  \begin{tabular}{ll}
  $\dfrac{1}{[h^\ast(1)]^2}$, & $|y| \leq 1$, \\
   & \\
  $\dfrac{1}{y^2[h^\ast(y)]^2}$, & $|y|>1$.
  \end{tabular}
  \right.
  $$
Since $f$ is continuous and $\lim_{|x|\rightarrow \infty} h^\ast(x)=\kappa^{-1}$, then $f$, $fh^\ast$ and $f(h^\ast)^2$ are integrable with respect to Lebesgue measure. On the other hand, $h$ is dominated by the symmetric function $h^\ast$, then the integrability of $fh^\ast$ and $f(h^\ast)^2$ imply 
  $$\int_{\mathbb{R}} f(y)h(y)dy<\infty, \quad \int_{\mathbb{R}} f(y)h(y)h(-y)dy < \infty.$$
Furthermore, since $h$ is subadditive and $f$ is symmetric, it follows
  $$\int_{\mathbb{R}} f(y) h(x-y) dy \leq \int_{\mathbb{R}} f(y)h(x) dy + \int_{\mathbb{R}} f(y)h(y)dy< \infty.$$
Thus, for $x\in \mathcal{H}$,
\begin{equation*}
U^\updownarrow f (x) = \int_{\mathbb{R}} f(y) u_0^\updownarrow(x,y) dy < \infty.
\end{equation*}
Finally, using the symmetry of $f$ and $h^\ast$ we obtain
\begin{equation*}
U^\updownarrow f(0) = \int_{\mathbb{R}} f(y)u^\updownarrow(0,y) dy = \int_{\mathbb{R}} f(y)(h^\ast(y)-h(y)) dy < \infty.
\end{equation*}
This concludes the proof.
\end{proof}

The Lemma \ref{indoflpaz} below states that any $x\in \mathcal{H}$ is regular for itself under $\mathbb{P}_x^\updownarrow$. The latter implies the existence of a continuous local time at point $x\in \mathcal{H}$ for the process $(X, \mathbb{P}_x^\updownarrow)$, see \cite[Theorem 3.12, p. 216]{Blumenthal-Getoor68}. We will denote by $(L^\updownarrow(x,t), t\geq 0)$ the local time at point $x$ aforementioned and by $\uptau^\updownarrow(x,t)$ the right continuous inverse of $L^\updownarrow(x,t)$, i.e.,
\begin{equation*}
\uptau^\updownarrow(x,t) = \inf\{s>0: L^\updownarrow(x,s) > t \}, \quad t\geq 0.
\end{equation*}
It is well known that $(\uptau^\updownarrow(x,t), t\geq 0)$ is a subordinator killed at an exponential random time independent of $\uptau^\updownarrow(x,\cdot)$ with Laplace exponent $\Phi^{x,\updownarrow}$ satisfying  
\begin{equation} \label{equ4343}
\mathbb{E}_x(e^{-q\uptau^\updownarrow(x,t)}) = e^{-t\Phi^{x,\updownarrow}(q)} = e^{-t/u_q^\updownarrow(x,x)}, \quad t\geq 0,
\end{equation}  
see e.g. \cite[Theorem 3.17, p. 218]{Blumenthal-Getoor68}. Furthermore, using the compensation formula in excursion theory it can be established that for any $q>0$,
\begin{equation} \label{equ4444}
\begin{split}
&\Phi^{x,\updownarrow}(q) = \frac{1}{u_q^\updownarrow(x,x)} =  n_x^\updownarrow ( \zeta>\mathbf{e}_q ) + a^xq \\ &= n_x^\updownarrow(\zeta=\infty) + \int_0^\infty (1-e^{-qt}) n_x^\updownarrow(\zeta\in dt) + a^xq,
\end{split}
\end{equation}
where $a^x$ satisfies
\begin{equation} \label{equ4545}
\int_0^t \mathbf{1}_{\{X_s=x\}} ds= a^xL^\updownarrow(x,t).
\end{equation}
By Remark \ref{remhast}, $\lim_{q\rightarrow 0}u_q^\updownarrow(x,x) = h^\ast(x)>0$, for $x\in \mathcal{H}$, then $(\uptau^\updownarrow(x,t), t\geq 0)$ is a subordinator killed at an exponential time with parameter $1/h^\ast(x)>0$. This also confirms the transiency of $(X, \mathbb{P}_x^\updownarrow)$, since by (\ref{equ4444}), there exists an excursion of infinite length.

To state the following lemma, we introduce additional notation. For every $x\in \mathbb{R}$, define $d_s^x = \inf\{u>s: X_t=x \}$, $g_s^x = \sup\{u\leq s: X_t=x\}$ and $G^x = \{g_u^x: g_u^x\neq d_u^x, u>0\}$.

\begin{lemma} \label{indoflpaz}
\begin{enumerate}
\item[(i)] For $x\in \mathcal{H}$, $x$ is regular for itself for $(X,\mathbb{P}_x^\updownarrow)$.

\item[(ii)] Let $\mathbf{e}_q$ be an exponential random variable with parameter $q>0$, independent of $(X, (\mathbb{P}_x^\updownarrow)_{ x\in \mathcal{H}})$. Then, for every $x\in \mathcal{H}$, the processes $(X_u, u<g_{\mathbf{e}_q}^x)$ and $X\circ k_{\mathbf{e}_q - g_{\mathbf{e}_q}^x}\circ \theta_{g_{\mathbf{e}_q}^x}$ are $\mathbb{P}_x^\updownarrow$ independent. Furthermore, their laws are characterized as follows: let $F$ and $H$ be measurable and bounded functionals, then
\begin{equation} \label{eqaux2-2}
\mathbb{E}_x^\updownarrow \left( F(X_u, u<g_{\mathbf{e}_q}^x) \right) = \mathbb{E}_x^\updownarrow \left( \int_0^\infty F(X_u,u<s) e^{-qs} dL^\updownarrow(x,s) \right) \left[ n_x^\updownarrow ( \zeta>\mathbf{e}_q ) + a^xq \right]
\end{equation}
and
\begin{equation} \label{equ4848} 
\mathbb{E}_x^\updownarrow \left(  H(X\circ k_{\mathbf{e}_q - g_{\mathbf{e}_q}^x}\circ \theta_{g_{\mathbf{e}_q}^x}) \right) = u_q^\updownarrow(x,x) \left[ n_x^\updownarrow \left( \int_0^\zeta H(\upepsilon_u, u<t) qe^{-qt} dt \right) + a^xqH(\bar{x}) \right],
\end{equation} 
where $a^x$ is the constant in (\ref{equ4545}).
\end{enumerate}
\end{lemma}
\begin{proof}
Let $x\in \mathcal{H}$. By Fatou's lemma and the definition of $\mathbb{P}_x^\updownarrow$, we have
\begin{eqnarray*}
\mathbb{P}_x^\updownarrow(T_x=0) &=& \liminf_{t\rightarrow 0} \mathbb{P}_x^\updownarrow(T_x\leq t) \\
&\geq& \frac{1}{h(x)} \mathbb{E}_x\left( \liminf_{t\rightarrow 0} \mathbf{1}_{\{T_x\leq t<T_0\}} h(X_t) \right) \\
&=& \frac{1}{h(x)} \mathbb{E}_x\left( \mathbf{1}_{\{T_x=0\}} \mathbf{1}_{\{T_0>0\}} h(X_0) \right) \\
&=& 1,
\end{eqnarray*}
where the latter equality was obtained using the facts that $\{x\}$ is regular for itself under $\mathbb{P}_x$ and $\mathbb{P}_x(T_0>0)=1$. This proves (i).

Before we prove (ii), we recall the following. Since $\uptau^\updownarrow(x,\cdot)$ is the inverse of the local time $(L^\updownarrow(x,t), t \geq 0)$ with Laplace exponent given by (\ref{equ4343}), then 
\begin{equation} \label{equ4646}
\mathbb{E}_x^\updownarrow \left( \int_0^\infty e^{-qt} dL^\updownarrow(x,t) \right) = \mathbb{E}_x^\updownarrow \left( \int_0^\infty e^{-q\uptau^\updownarrow(x,t)} dt \right) = \int_0^\infty \mathbb{E}_x^\updownarrow ( e^{-q\uptau^\updownarrow(x,t)} ) dt = u_q^\updownarrow(x,x). 
\end{equation}
We will denote $\bar{x}$ the path which is identically equal to $x$ and with lifetime zero. Thus, for $F$ and $H$ measurable and bounded functionals, using the compensation formula in excursion theory (see e.g. \cite{Bertoin}, \cite{Maisonneuve75}), it follows
\begin{align}
\nonumber
&\mathbb{E}_x^\updownarrow \left( F(X_u, u<g_{\mathbf{e}_q}^x) H(X\circ k_{\mathbf{e}_q - g_{\mathbf{e}_q}^x}\circ \theta_{g_{\mathbf{e}_q}^x}) \right) \\ \nonumber
&= \mathbb{E}_x^\updownarrow \left( \sum_{s\in G^x} F(X_u,u<s) e^{-qs} \int_s^{d_s} H(X\circ k_{t-s} \circ \theta_s) qe^{-q(t-s)} dt \right) \\ \nonumber
&\quad + \mathbb{E}_x^\updownarrow \left( \int_0^\infty F(X_u,u<t )H(\bar{x}) qe^{-qt} \mathbf{1}_{\{X_t=x\}} dt \right) \\ \label{equ4747}
&= \mathbb{E}_x^\updownarrow \left( \int_0^\infty F(X_u,u<s) e^{-qs} dL^\updownarrow(x,s) \right) \left[ n_x^\updownarrow \left( \int_0^\zeta H(\upepsilon_u, u<t) qe^{-qt} dt \right) + a^xqH(\bar{x}) \right], 
\end{align}
where $a^x$ is the constant in (\ref{equ4545}). Taking $H\equiv 1$ in (\ref{equ4747}), it follows
\begin{equation*}
\mathbb{E}_x^\updownarrow \left( F(X_u, u<g_{\mathbf{e}_q}^x) \right) = \mathbb{E}_x^\updownarrow \left( \int_0^\infty F(X_u,u<s) e^{-qs} dL^\updownarrow(x,s) \right) \left[ n_x^\updownarrow ( \zeta>\mathbf{e}_q ) + a^xq \right].
\end{equation*}
In the same way, if we take $F\equiv 1$ in (\ref{equ4747}) and we use (\ref{equ4646}), we can obtain 
\begin{equation*} 
\mathbb{E}_x^\updownarrow \left(  H(X\circ k_{\mathbf{e}_q - g_{\mathbf{e}_q}^x}\circ \theta_{g_{\mathbf{e}_q}^x}) \right) = u_q^\updownarrow(x,x) \left[ n_x^\updownarrow \left( \int_0^\zeta H(\upepsilon_u, u<t) qe^{-qt} dt \right) + a^xqH(\bar{x}) \right].
\end{equation*}
The latter two displays are (\ref{eqaux2-2}) and (\ref{equ4848}), respectively.

Finally, by (\ref{equ4444}), $u_q^\updownarrow(x,x) = [ n_x^\updownarrow ( \zeta>\mathbf{e}_q ) + a^xq ]^{-1}$. Using this fact, (\ref{eqaux2-2}) and (\ref{equ4848}), we conclude 
  $$\mathbb{E}_x^\updownarrow \left( F(X_u, u<g_{\mathbf{e}_q}^x) H(X\circ k_{\mathbf{e}_q - g_{\mathbf{e}_q}^x}\circ \theta_{g_{\mathbf{e}_q}^x}) \right) = \mathbb{E}_x^\updownarrow \left( F(X_u, u<g_{\mathbf{e}_q}^x) \right) \mathbb{E}_x^\updownarrow \left(  H(X\circ k_{\mathbf{e}_q - g_{\mathbf{e}_q}^x}\circ \theta_{g_{\mathbf{e}_q}^x}) \right).$$
This shows the independence property in (ii).
\end{proof}

Now, we will prove that the drift coefficient in (\ref{equ4444}) does not depend on $x$, and is equal to $\delta$.

\begin{lemma} \label{axanddelta}
Let $\delta$ be the drift coefficient of the inverse local time at the point zero for the L\'evy process $(X,\mathbb{P})$. Then for all $x\in \mathcal{H}$, $\mathbb{P}_x^\updownarrow$-a.s., $\int_0^t \mathbf{1}_{\{X_s=0\}} ds = \delta L^\updownarrow(x,t)$. That is, $a^x=\delta$, for all $x\in \mathcal{H}$.
\end{lemma}
\begin{proof}
If both $\delta$, $a^x$ are zero, the claim holds. Suppose that $a^x\neq 0$. Using (\ref{equ4646}), the definition of $a^x$ and $\mathbb{P}_x^\updownarrow$, we obtain
\begin{eqnarray} \nonumber
a^xu_q^\updownarrow(x,x) &=&  \mathbb{E}_x^\updownarrow \left( \int_0^\infty e^{-qt} d[a^xL^\updownarrow(x,t)] \right) \\ \nonumber
&=& \int_0^\infty \mathbb{E}_x^\updownarrow \left( \mathbf{1}_{\{X_t=x\}} \right) e^{-qt} dt \\ \nonumber
&=& \int_0^\infty \frac{1}{h(x)}\mathbb{E}_x \left( \mathbf{1}_{\{X_t=x\}} h(X_t) \mathbf{1}_{\{T_0>t\}} \right) e^{-qt} dt \\ \label{equ4949}
&=& \mathbb{E}_x \left( \int_0^\infty \mathbf{1}_{\{T_0>t\}} e^{-qt} \mathbf{1}_{\{X_t=x\}} dt \right). 
\end{eqnarray}
Using that $(X,\mathbb{P}_x)$ is equal in distribution to $(X+x,\mathbb{P})$, the definition of $\delta$ and the symmetry of $h_q^\ast(x)$, it follows that the right-hand side in (\ref{equ4949}) is   
\begin{equation*}
\mathbb{E} \left( \int_0^\infty \mathbf{1}_{\{T_{-x}>t\}} e^{-qt} \mathbf{1}_{\{X_t=0\}} dt \right) =  \delta \mathbb{E} \left( \int_0^\infty \mathbf{1}_{\{T_{-x}>t\}} e^{-qt} dL_t \right) = \delta h_q^\ast(x). 
\end{equation*}
To conclude the proof recall that $h_q^\ast(x) = u_q^\updownarrow(x,x)$. 
\end{proof}

\begin{proof}[Proof of Proposition \ref{relexcmeasure}]
Let $H:\mathcal{D}^0\rightarrow \mathbb{R}$ be a bounded and measurable functional. To simplify we write $X^q$ for the path $X\circ k_{\mathbf{e}_q - g_{\mathbf{e}_q}^x}\circ \theta_{g_{\mathbf{e}_q}^x}$. Using the definition of $\mathbb{P}_x^\updownarrow$, we obtain 
\begin{equation} \label{equ5050}
h(x) \mathbb{E}_x^\updownarrow \left( H(X^q) \right) = \mathbb{E}\left( \int_0^\infty H((X+x)\circ k_{t-g_t} \circ \theta_{g_t}) h(X_t+x) \mathbf{1}_{\{T_{-x}>t\}} qe^{-qt} dt\right). 
\end{equation}
We note that $\mathbf{1}_{\{T_{-x}>t\}} = \mathbf{1}_{\{T_{-x}\circ \theta_{g_t}>t-g_t\}} \mathbf{1}_{\{T_{-x}>g_t\}}$ and $h(X_t+x) = h((X_{t-g_t}+x)\circ \theta_{g_t})$. Then, with the help of the compensation formula in excursion theory (\cite{Bertoin}, \cite{Maisonneuve75}), the right-hand side in (\ref{equ5050}) can be written as
\begin{align} \nonumber
\mathbb{E} & \left(  \int_0^\infty \mathbf{1}_{\{T_{-x}>g_t\}} H((X+x)\circ k_{t-g_t} \circ \theta_{g_t}) h((X_{t-g_t}+x)\circ \theta_{g_t}) \mathbf{1}_{\{T_{-x}\circ \theta_{g_t}>t-g_t\}} qe^{-qt} dt\right) \\ \nonumber
&= \mathbb{E} \left( \sum_{s\in G} \mathbf{1}_{\{T_{-x}>s\}} \int_s^{d_s} H((X+x)\circ k_{t-s} \circ \theta_s) h((X_{t-s}+x)\circ \theta_s) \mathbf{1}_{\{T_{-x}\circ \theta_s>t-s\}} qe^{-qt} dt\right) \\ \nonumber
&\quad + \mathbb{E} \left( \int_0^\infty \mathbf{1}_{\{T_{-x}>t\}} H(\bar{x}) h(x) qe^{-qt} \mathbf{1}_{\{X_t=0\}} dt\right) \\ \nonumber
&= \mathbb{E}\left( \int_0^\infty \mathbf{1}_{\{T_{-x}>t\}} e^{-qt}dL_t \right) n \left( \int_0^\zeta H(\upepsilon_u+x, u<t)h(X_t+x)\mathbf{1}_{\{T_{-x}>t\}} qe^{-qt} dt \right)  \\ \nonumber
&\quad + q\delta H(\bar{x})h(x) \mathbb{E}\left( \int_0^\infty \mathbf{1}_{\{T_{-x}>t\}} e^{-qt}dL_t \right) \\ \label{equ5151}  
&= h_q^\ast(x) \left[ n \left( \int_0^\zeta H(\upepsilon_u+x, u<t) h(X_t+x)\mathbf{1}_{\{T_{-x}>t\}} qe^{-qt} dt \right) + q\delta H(\bar{x})h(x) \right],
\end{align}
where $\delta$ is such that $\int_0^t\mathbf{1}_{\{X_s=0\}} = \delta L_t$ under $\mathbb{P}$. Using Lemma \ref{axanddelta} and $h^\ast(x)=u_q^\updownarrow(x,x)$ in (\ref{equ5151}), we verify 
\begin{equation} \label{equ5252}
\mathbb{E}_x^\updownarrow \left( H(X^q) \right)= u_q^\updownarrow(x,x) \left[ \frac{1}{h(x)} n \left( \int_0^\zeta H(\upepsilon_u+x, u<t) h(X_t+x)\mathbf{1}_{\{T_{\{-x\}}>t\}} qe^{-qt} dt \right) + a^xqH(\bar{x})\right].
\end{equation}
Comparing (\ref{equ5252}) with (\ref{equ4848}), the result follows.  
\end{proof}

\section{Two examples} \label{examples}

The expression (\ref{equ2121}) in Theorem \ref{propofh} (i) allows us to compute explicitly the function $h$ in the particular case when $(X,\mathbb{P})$ is an $\alpha$-stable process.

\begin{example}
Suppose that $(X,\mathbb{P})$ is an $\alpha$-stable process. Then, $(X,\mathbb{P})$ satisfies \textbf{H.1} and \textbf{H.2} if and only if $\alpha\in (1,2]$. It is well known that the resolvent density of Brownian motion is $u_q(x) = (\sqrt{2q})^{-1}e^{-\sqrt{2q}|x|}$, hence $h(x)=\lim_{q\rightarrow 0} [u_q(0) - u_q(-x)] =|x|$. Now, let $\alpha \in (1,2)$. In this case the function $h$ takes the following form
\begin{equation} \label{halfa}
h(x) = K(\alpha)(1-\beta\sign(x))|x|^{\alpha-1},
\end{equation}
where
\begin{equation*}
K(\alpha) = \frac{\Gamma(2-\alpha)\sin(\alpha\pi/2)}{c\pi(\alpha-1)(1+\beta^2\tan^2(\alpha\pi/2))}
\end{equation*}
and
\begin{equation} \label{equ5353}
c = -\frac{(c^++c^-)\Gamma(2-\alpha)}{\alpha(\alpha-1)}\cos(\alpha\pi/2), \quad \beta = \frac{c^+-c^-}{c^++c^-}.
\end{equation}
Notice that if $\beta=1$ ($c^-=0$), then $h(x)=0$, for all $x\geq 0$, and if $\beta=-1$ ($c^+=0$), then $h(x)=0$, for all $x\leq 0$.

Before obtaining (\ref{halfa}) we will verify hypotheses \textbf{H.3} and {\bf H.4}, respectively. Recall that the characteristic exponent of $(X,\mathbb{P})$ can be written as
\begin{equation*}
\psi(\lambda) = c|\lambda|^\alpha(1-i\beta\sign(\lambda)\tan(\alpha\pi/2)),
\end{equation*}
where $c$ and $\beta$ are as in (\ref{equ5353}). Notice that {\bf H.3} is satisfied because the integral
  \[
  \int_\mathbb{R} \frac{1}{q+c|\lambda|^\alpha} d\lambda
  \]
is finite for all $q>0$. On the other hand, using the expression of $\psi(\lambda)$ we obtain
\begin{equation*}
Re\left( \frac{1-e^{i\lambda x}}{\psi(\lambda)} \right) = \frac{ 1-\cos(\lambda x) + \beta\tan(\alpha\pi/2)\sign(\lambda)\sin(\lambda x) }{ c|\lambda|^\alpha(1 + \beta^2\tan^2(\alpha\pi/2)) }.
\end{equation*}
From the latter equaility, we have
\begin{align*}
\left|  Re\left( \frac{1-e^{i\lambda }}{\psi(\lambda)} \right) \right| &\leq \frac{1}{c} \left( \frac{ 1-\cos(\lambda ) }{ |\lambda|^\alpha(1 + \beta^2\tan^2(\alpha\pi/2))} +  \frac{ |\beta\tan(\alpha\pi/2)||\sin(\lambda )| }{ |\lambda|^\alpha(1 + \beta^2\tan^2(\alpha\pi/2)) } \right) \\
&\leq \frac{1}{c} \left( \frac{ 2( 1\wedge |\lambda|^2) }{ |\lambda|^\alpha } + \frac{ 1 \wedge |\lambda| }{ |\lambda|^\alpha } \right).
\end{align*}
Now, since $1<\alpha<2$, 
  \[
  \int_\mathbb{R} \frac{  1\wedge |\lambda|^2 }{ |\lambda|^\alpha } d\lambda = 2 \int_0^\infty \frac{ 1\wedge \lambda^2  }{ \lambda^\alpha } d\lambda = \frac{4}{(3-\alpha)(\alpha-1)}
  \]
and
  \[
  \int_\mathbb{R} \frac{ 1 \wedge |\lambda| }{ |\lambda|^\alpha } d\lambda = 2 \int_0^\infty \frac{ 1 \wedge \lambda }{ \lambda^\alpha } d\lambda = \frac{2}{(2-\alpha)(\alpha-1)}
  \]
are finites. Thus, \textbf{H.4} is verified.

Next, we obtain (\ref{halfa}). Since the function $\sign(\lambda)\sin(\lambda x)$ as function of $\lambda$ is even, we have
\begin{equation*}
 \int_{-\infty}^\infty Re\left( \frac{1-e^{i\lambda x}}{\psi(\lambda)} \right) d\lambda = \frac{1}{c(1 + \beta^2\tan^2(\alpha\pi/2))}\left( 2\pi h^s(x) + 2\beta\tan(\alpha\pi/2) \int_0^\infty \frac{\sin(\lambda x)}{\lambda^\alpha} d\lambda \right),
\end{equation*}
where $h^s(x)$ is the function $h$ obtained in the symmetric case (see example 1.1 in \cite{Yano10}), namely
\begin{equation*}
h^s(x) = \frac{ \Gamma(2-\alpha) }{ \pi(\alpha-1) } \sin(\alpha\pi/2)|x|^{\alpha-1}.
\end{equation*}
On the other hand, with the help of formulae (14.18) of Lemma 14.11 in \cite{Sato}, we obtain
\begin{equation*}
\int_0^\infty \frac{\sin(\lambda x)}{\lambda^\alpha} d\lambda = \sign(x)|x|^{\alpha-1} \int_0^\infty \frac{\sin u}{u^\alpha} du = -\frac{ \Gamma(2-\alpha) }{ \alpha-1 } \cos(\alpha\pi/2) \sign(x)|x|^{\alpha-1}.
\end{equation*}
The latter equality implies (\ref{halfa}).

For reference, we point out that similar calculations are performed in \cite[Lemma 5.4]{Kogan-Marcus-Rosen11} to determine $u_{T_0}(x,y):=\mathbb{E}_x(L_{T_0}^y)$, where $L_t^y$ is the local time at the point $y$ for the process $(X, \mathbb{P}_x)$. The function $u_{T_0}(x,y)$ is related to $h^\ast$ through the formula $u_{T_0}(x,x)=\mathbb{E}_x(L_{T_0}^x)=\mathbb{E}(L_{T_{-x}}) = h^\ast(-x) = h^\ast(x)$.

\end{example}

Recall that $L_{T_x}$ is an exponential random variable with parameter $[h^\ast(x)]^{-1}$, then by (\ref{equ3434}), in the case when $(X,\mathbb{P})$ is an $\alpha$-stable process with $\alpha\in(1,2]$, $L_{T_x}$ is an exponential random variable with parameter $1/[2K(\alpha)|x|^{\alpha-1}]$.

As was mentioned in the above example, depending on if $\beta=1$ or $\beta=-1$, the function $h$ vanishes at some points different from zero. The case $\beta=-1$ corresponds to a spectrally negative $\alpha$-stable process. In the following example an expression for the function $h$ is obtained for spectrally negative L\'evy processes satisfying the hypothesis \textbf{H1} and \textbf{H2}. In \cite{Pardo-Perez-Rivero15} is used this example to give a detailed description of the excursion measure away from zero for spectrally negative L\'evy processes.

\begin{example} \label{spectrallyexample}
Let $(X,\mathbb{P})$ be a spectrally negative L\'evy process, i.e., the L\'evy measure of $(X,\mathbb{P})$ satisfies $\pi(0,\infty) = 0$. Suppose that {\bf H.1} and {\bf H.2} hold. Let $\Psi$ be the Laplace exponent of the process $(X,\mathbb{P})$, i.e., 
\begin{equation*}
\Psi(\lambda) = \log(\mathbb{E}[e^{\lambda X_1}]) = -\psi(-i\lambda) = -a\lambda + \frac{\sigma^2}{2}\lambda^2 + \int_{(-\infty, 0)} (e^{\lambda x} - 1 - \lambda x\mathbf{1}_{\{x > -1\}}) \pi(dx).
\end{equation*}
It is well known that $\Psi'(0+) = \mathbb{E}(X_1) \in [-\infty, \infty)$ determines the long run behaviour of $X$. To be precise, if $\Psi'(0+) > 0$ then $\lim_{t\to \infty} X_t=\infty$, if $\Psi'(0+)<0$ then $\lim_{t\to \infty} X_t = -\infty$, while if $\Psi'(0+) = 0$ the process $X$ oscillates.

Now, for $q\geq 0$, let $\Phi(q)$ be the largest root of the equation $\Psi(\lambda) = q$, i.e., 
\begin{equation*}
\Phi(q) = \sup\{\lambda\geq 0: \Psi(\lambda) = q\}.
\end{equation*}
For the spectrally negative L\'evy processes $(X,\mathbb{P})$ with Laplace exponent $\Psi$, we consider the $q$-scale functions $\{W^{(q)}, q\geq 0\}$, the family of functions satisfying the following: for every $q\geq 0$, $W^{(q)}(x)=0$, for $x<0$ and $W^{(q)}(x)\geq 0$, for $x\geq 0$. Furthermore, $W^{(q)}$ is determined by its Laplace transform in the following way
\begin{equation*}
\int_0^\infty e^{-\theta x} W^{(q)}(x)dx = \frac{ 1 }{ \Psi(\theta) - q}, \quad \theta > \Phi(q).
\end{equation*}
For a complete account on $q$-scale functions for spectrally negative L\'evy processes see \cite{Kuznetsov-Kyprianou-Rivero12}.

An important fact on spectrally L\'evy processes is that the resolvent density, $u_q$, can be written in terms of the $q$-scale function $W^{(q)}$ as follows 
\begin{equation*}
u_q(x) = \Phi'(q)e^{-\Phi(q) x} - W^{(q)}(-x), \quad q>0, x\in\mathbb{R}.
\end{equation*}
Furthermore, if $(X,\mathbb{P})$ has unbounded variation, $W^{(q)}(0)=0$, see Corollary 8.9 and Lemma 8.6 in \cite{Kyprianou06} for details. Now, Corollary 5 in Chapter VII in \cite{Bertoin} ensures that under {\bf H.1} and {\bf H.2}, $X$ has unbounded variation. †The latter facts imply, 
\begin{equation*}
h_q(x) = u_q(0)- u_q(-x) = \Phi'(q)(1-e^{\Phi(q) x}) + W^{(q)}(x), \quad q>0, x\in \mathbb{R}.
\end{equation*}
Thus, letting $q\to 0$, we obtain
\begin{equation*}
h(x) = \left\{
\begin{array}{ll}
\dfrac{1}{\Psi'(\Phi(0)+)} (1-e^{\Phi(0)x}) + W(x), & \text{if } \displaystyle{\lim_{t\to\infty }X_t = -\infty}, \\
\dfrac{-x}{\Psi''(0+)} + W(x), & \text{if } \displaystyle{\limsup_{t\to\infty }X_t = -\liminf_{t\to \infty} X_t = \infty}, \\
W(x), & \text{if } \displaystyle{\lim_{t\to\infty }X_t = \infty},
\end{array}
\right.
\end{equation*}
where $W=W^{(0)}$. So the function $h$ on this setting can be defined only with the hypotheses {\bf H.1} and {\bf H.2}. Observe that when $X$ oscillates and $\Psi''(0+) = \infty$ we have $h(x)=W(x)$ and hence $h(x) = 0$ for $x<0$.
\end{example}

\bigskip

\noindent{\bf Acknowledgements}. I would like to thank my advisors, Lo\"ic Chaumont and V\'ictor Rivero, for guiding me through the elaboration of this work. I would like to thank as well to the CNRS-CONACYT International Laboratory Solomon Lefschetz (LAISLA) by the support given during this research. Ce travail a b\'en\'efici\'e d'une aide de l'Agence Nationale de la Recherche portant la r\'ef\'erence ANR-09-BLAN-0084-01.

\bibliographystyle{abbrv}
\bibliography{ref2}

\end{document}